\documentclass[preprint]{elsarticle}
\usepackage[utf8]{inputenc}
\usepackage{amsfonts}
\usepackage{amsmath}
\usepackage{amssymb}
\usepackage{amsthm}
\usepackage{rotating}
\usepackage{graphicx}
\usepackage{subfigure}
\usepackage{color}
\usepackage{tikz-cd}
\usetikzlibrary{arrows,backgrounds}
\usepackage{verbatim}
\usepackage{hyperref}
\usepackage[]{algorithm2e}
\usepackage{bm}

\newtheorem{thm}{Theorem}[section]
\newtheorem{remark}{Remark}[section]

\begin{document}

\begin{frontmatter}

\title{Enforcing exact physics in scientific machine learning: a data-driven exterior calculus on graphs}
\author[1]{Nathaniel Trask}
\ead{natrask@sandia.gov}
\author[2]{Andy Huang}
\author[3]{Xiaozhe Hu}

\cortext[cor1]{Corresponding author}
\address[1]{Center for Computing Research, Sandia National Laboratories}
\address[2]{Radiation and Electrical Science, Sandia National Laboratories}
\address[3]{Department of Mathematics, Tufts University}

\begin{abstract}
    As traditional machine learning tools are increasingly applied to science and engineering applications, physics-informed methods have emerged as effective tools for endowing inferences with properties essential for physical realizability. While promising, these methods generally enforce physics weakly via penalization. To enforce physics strongly, we turn to the exterior calculus framework underpinning combinatorial Hodge theory and physics-compatible discretization of partial differential equations (PDEs). Historically, these two fields have remained largely distinct, as graphs are strictly topological objects lacking the metric information fundamental to PDE discretization. We present an approach where this missing metric information may be learned from data, using graphs as coarse-grained mesh surrogates that inherit desirable conservation and exact sequence structure from the combinatorial Hodge theory. The resulting data-driven exterior calculus (DDEC) may be used to extract structure-preserving surrogate models with mathematical guarantees of well-posedness. The approach admits a PDE-constrained optimization training strategy which guarantees machine-learned models enforce physics to machine precision, even for poorly trained models or small data regimes. We provide analysis of the method for a class of models designed to reproduce nonlinear perturbations of elliptic problems and provide examples of learning $H(div)/H(curl)$ systems representative of subsurface flows and electromagnetics. 
\end{abstract}


\end{frontmatter}

\section{Introduction}

Network models of physical systems are ubiquitous throughout the sciences and engineering. The electronic circuit models typically learned in first year undergraduate physics may similarly be used to describe fluid, mechanical or heat transfer systems in corresponding hydraulic circuit, mass-spring-damper, or thermal circuit analogies, respectively \cite{weyl-1923,wilf-hilary-1972,firestone1933new,koenig1960linear,bloch1945electromechanical,smith2002synthesis,chen2015electrical,oh2012design}. Such network models represent discrete representations of conservation laws corresponding to classical control volume analysis and their physical properties are intricately tied to the underlying graph topology \cite{smale-1972,bamberg-sternberg-1991}, allowing modeling of a number of thermodynamic principles \cite{breedveld1985multibond}.  In contemporary machine learning (ML), graph neural networks (GNNs) attach a more "black-box" message passing model to a network, inferring “object-relation”/causal relationships from unstructured data to obtain data representations and model pair-wise interactions \cite{gori2005new,hamilton2017inductive,hamilton2017representation,battaglia2016interaction,chang2016compositional}. Generally however, applications of GNNs have not focused upon preserving physical structures in their network topology, although some recent works have considered how the interplay between microstructure and mechanics can be encoded via graphs \cite{vlassis2020geometric}.

In this work we present a framework to extract efficient data-driven network models which exactly preserve desirable mathematical structures of the underlying physics. The framework we introduce is general, but we focus on a particular application in which we assume access to high-fidelity data and aim to learn a low dimensional network as an efficient structure-preserving surrogate model. The resulting model bears similarities to reduced-order models (ROMs), however the control volume analysis pursued in the current approach may provide advantages in preserving structural properties that have proven challenging for the variational derivation underpinning ROMs \cite{quarteroni2015reduced,ohlberger2015reduced}.

Solutions of PDEs generally rely upon computational meshes partitioning space into disjoint curvilinear cells. The shared topological structure offered by both networks and meshes allow a unified analysis in terms of the exterior calculus. For PDEs, the discrete exterior calculus (DEC) and related spatially compatible discretizations encapsulate a range of so-called mimetic discretizations providing discrete solutions mimicking physical properties of the continuous problem \cite{desbrun-hirani-leok-marsden-2005,arnold2007compatible,arnold2018finite,bochev2006principles}. These methods generally preserve conservation properties and spectral representations of operators, provide a coordinate-free means of prescribing physics on manifolds, and allow handling of the non-trivial null-spaces required in electromagnetics. In topological data analysis, combinatorial Hodge theory has emerged as a tool for analyzing flows on graphs along with their spectral and homological properties, e.g.,~\cite{MuhammadEgerstedt2006,JiangLimYaoYe2011, MaleticRajkovic2012,BarbarossaSardellitti2020, BattistonCencettiIacopiniLatoraLucasPataniaYoungPetri2020,Lim2020,SchaubBensonHornLippnerJadbabaie2020,TorresBianconi2020}. These techniques are supported by a graph calculus providing generalizations of gradient, curl, and divergence operators admitting interpretation as discrete exterior derivatives. As graphs are purely topological however, this graph calculus lacks any metric information, and the operators do not converge in a meaningful way to the familiar vector calculus necessary to model physical systems.

The key observation of the current work is that one may exploit the shared structure of graphs and meshes and use data to endow the graph calculus with "missing" metric information, in the process learning a data-driven exterior calculus (DDEC). Conservation laws are encoded onto the graph via trainable div/curl operators satisfying a Stokes theorem, while "black-box" DNNs may parameterize fluxes. This guarantees that the mathematical structures related to conservation and exact sequence properties are independent of the errors incurred during training. Consequently, this allows us to prove a number of properties of the resulting model, independent of the quantity of available data or particular local minimizer found during training.

The incorporation of physical principles into scientific machine learning tasks has surged in recent years, as it has become apparent that "off-the-shelf" ML tools often fail to provide robust predictions for science and engineering applications \cite{baker2019workshop}. Physics-informed approaches have achieved a range of successes by introducing regularizers that penalize deviations from physical properties \cite{raissi2017physics,raissi2019physics,lagaris1998artificial}. While simple to implement and effective for a range of problems, 
physics in these cases are only enforced to within optimization error and may hold to a relatively coarse tolerance (for open problems regarding their training, see e.g. \cite{wang2020understanding,wang2020and}). Some applications however require constraints to hold to machine precision; e.g. in electrodynamics it is critical that the solenoidal constraint hold to machine precision to handle the involution condition \cite{barth2006role,dafermos1986quasilinear}, and in forward modeling contexts compatible/mimetic discretizations provide approximations guaranteeing such properties hold by construction \cite{yee1966numerical,nedelec1980mixed,bossavit1998computational,bochev2001matching}. The current work provides an analogous means of designing machine learning architectures which enforce physics by construction, therefore removing the need for physics-informed regularizers.


\subsection{Paper organization}
We first recall necessary exterior calculus fundamentals before introducing our data-driven exterior calculus. While the theory is abstract, we focus on applying it to learn nonlinear perturbations of $div-grad$ and $div-curl$ model problems as canonical examples of physics requiring structure preservation. We next provide numerical analysis, establishing conditions under which the learned model has a unique solution. We establish that the data-driven exterior calculus inherits the desirable properties of the graph calculus and use them to analyze the well-posedness of a class of nonlinear elliptic problems. A necessary implementation question is how to obtain a graph to define the model upon. We provide a specific example considering coarse-graining of a high-fidelity mesh, associated with either a finite element simulation or histograms binning experimental data, and show how the relevant commuting diagrams are preserved under coarsening. Finally, we introduce a PDE-constrained optimizer to fit the model to data, allowing enforcement of physics exactly via equality constraint. The numerical analysis implies that the forward problem associated with the equality constraint is always well-posed given mild conditions on the architecture. Finally, we provide several numerical results demonstrating how one may learn efficient physics-preserving surrogates from high-fidelity data.

\section{The graph exterior calculus}\label{sec:graphcalc}

We recall first the graph calculus, which serves as the foundation for DDEC. Let $\mathbf{N} = \left\{n_i\right\}_{i=1}^{N_N}$ denote a set of nodes. We embed $\mathbf{N}$ in $\mathbb{R}^d$ by associating with each node a unique position $\mathbf{p}_i \subset \mathbf{R}^d$, $i \in 1,\dots,N_N$. We define a $k$-clique as an ordered tuple consisting of $k$ nodes, i.e. $t_k = [n_1,...,n_k]$. A $k$-clique has positive orientation if $\pi = \left\{i_1,...,i_k\right\}$ is an even permutation of $\left\{1,...,k\right\}$ and negative otherwise. Via the embedding of the graph, we may associate with each $k$-clique the $(k-1)$-simplex defined as the convex hull of the vertices $s_{k}=[\mathbf{p}_1,...,\mathbf{p}_k]$, for which the $k \leq d$ distinct points span a $k$-dimensional hyperplane. A $k$-chain $c_k$ may then be defined as a linear combination of (k+1)-cliques, and we denote the set of $k$-chains by $C_k$. One may introduce a \textit{boundary operator} $\partial_k : C_{k+1} \rightarrow C_{k}$ defined via

\begin{equation}\label{eq:bndry}
  \partial_k [n_1,...,n_k] = \sum_i^k (-1)^{i-1} [n_1,...,n_{i-1},\widehat{n_{i}},n_{i+1},n_k],
\end{equation}
where $\hat{\cdot}$ denotes an omitted entry, and which satisfies the property $\partial_{k-1} \circ \partial_k = 0$. When the dependence upon $k$ is clear, we will sometimes write the coboundary simply as $\partial$ - we will adopt this convention for similar operators throughout. With these definitions in hand we may finally introduce the chain complex as the following exact sequence pairing $k$-chains and boundary operators.

\begin{equation}
  \begin{tikzcd}\label{eq:chaincomplex}
    C_0 & \arrow[l, "\partial_0"] C_1 & \arrow[l, "\partial_1"] \dots & \arrow[l, "\partial_{d-1}"]  C_d 
  \end{tikzcd}
\end{equation}
with the standard convention that $\partial_{-1}$ maps $C_0$ to the empty set.

Note that in this graph context, the specific realization of the chain complex may be qualitatively different from the DEC setting. Traditionally in compatible discretizations, the complex is realized by partitioning the domain of interest into a collection of disjoint simplices to obtain $C_d$, and then deriving lower dimensional mesh entities $C_k$, $k<d$ via the boundary operator. In contrast, for the graph setting one may obtain overlapping simplices, in the sense that given unique $t_k,t_k' \in C_l$, $t_k \cap t_k' \neq \emptyset$. We will later discuss details regarding specific choice of chain complex, but for now keep the presentation abstract.

We next associate real numbers with the graph entities constituting the chain complex. For each set of chains $C_k$, we introduce the dual set of cochains $C^k$ consisting of linear functionals acting on $C_k$. Given $\phi \in C^k$, we denote the value associated with the $k$-chain $t_{i_1 i_2 ... i_k}$ via the shorthand $\phi_{i_1 i_2 ... i_k} := \phi(t_{i_1 i_2 ... i_k})$. Note that cochains inherit the orientation of the underlying chains, e.g. $\phi_{ij} = -\phi_{ji}$ via the definition of $\pi$. Introducing the \textit{coboundary operator} $\delta_k: C^k \rightarrow C^{k+1}$, we next arrive at the following cochain complex

\begin{equation}
  \begin{tikzcd}\label{eq:cochaincomplex}
    C^0  \arrow[r, "\delta_0"] & C^1 \arrow[r, "\delta_1"] & \dots \arrow[r, "\delta_{d-1}"]&  C^d 
  \end{tikzcd}
\end{equation}

We will formally denote the pairing between boundary and coboundary operators via the inner product
\begin{equation}
    \langle\phi,\partial_k t\rangle = \langle\delta_k \phi, t\rangle.
\end{equation}
In the traditional DEC setting, one would arrive at a definition of the coboundary via the generalized Stokes theorem, defined by the dual pairing of $\delta_k$ and $\partial_k$ via $\int_{\partial \omega} w = \int_\omega \delta w$. In the graph setting, we identify $\delta_k$ algebraically as the adjoint of the matrix representing $\partial_k$. For example, this gives rise to the following combinatorial gradient, assigning to the $2$-clique $t_{ij}$ the function
\begin{equation}\label{eq:combd0}
  \delta_0 \phi_{ij} = \phi_j - \phi_i.
\end{equation}
Similarly we may obtain the combinatorial curl by assigning to the 3-clique $t_{ijk}$ the function
\begin{equation}\label{eq:combd1}
  \delta_1 \phi_{ijk} = \phi_{ij} + \phi_{jk} + \phi_{ki}.
\end{equation}
One may easily see that $\delta_1 \circ \delta_0 = 0$. For the purposes of this work, we will not require $\delta_k$ for $k>1$, however the definition extends naturally to $1 < k\leq d$, and one may show that $\delta_k \circ \delta_{k-1} = 0$. In this manner, the coboundary operator inherits the exact sequence property of the boundary operator.

We next let $(\cdot,\cdot)_k$ denote an inner product mapping $C^k \times C^k \rightarrow \mathbb{R}$. 
This inner product induces a \textit{codifferential operator} $\delta^*_k : C^{k+1} \rightarrow C^{k}$ via the pairing $(v,\delta^*_k u)_k = (\delta_k v,u)_{k+1}$.  In this manner, the careful choice of $(\cdot,\cdot)_k$ will endow the codifferential with desirable approximation properties, however we note that independent of the choice of inner products the codifferential again inherits the exact sequence properties of the coboundary operator so that $\delta^*_{k-1} \circ \delta^*_{k} = 0$. This follows trivially from the definition, so that for all $v\in C^{k-2}$ and all $u\in C^{k}$ 
\begin{equation}
    (v,\delta^*_{k-1} \delta^*_k u)_{k-2} = (\delta_{k-1} v, \delta^*_k u)_{k-1} = (\delta_k \delta_{k-1} v, u)_{k} = 0.
\end{equation}

Finally, we will contrast how the choice of inner product typically used in the graph exterior calculus precludes the use of graph boundary/coboundary operators in discretizing PDE. In the graph context, one selects as $(\cdot,\cdot)_k$ the $\ell_2$ inner-product: 
\begin{equation}
(x,y)_k = \underset{{i \in \text{dim}(C^k)}}{\sum} x_i y_i.
\end{equation}
And its induced norm is denoted by $\| \cdot \|_k$. In this case, the codifferential $\delta^*_k$ may be identified as the transpose of the matrix associated with the coboundary $\delta_k$. For the remainder of this work, we will assume $\ell_2$ inner products in the definition of the codifferential unless otherwise specified, so that $\delta_0^*$ and $\delta_1^*$ correspond to the graph divergence and curl.


Besides the properties that $\delta_k \circ \delta_{k-1} = 0$ and $\delta_{k-1}^* \circ \delta_k^* = 0$, the Hodge-Laplacian is defined as follows 
\begin{equation}\label{def:hodge-laplacian}
L_k = \delta_{k-1} \delta_{k-1}^* + \delta_k^* \delta_k.
\end{equation}
From the combinatorial Hodge theorem, the combinatorial Laplacian decomposes the space $C^k$ as follows
\begin{equation}\label{eqn:hodge-decomp}
C^k = \operatorname{im}(\delta_{k-1}) \bigoplus \operatorname{ker}(L_k) \bigoplus \operatorname{im}(\delta_{k}^*),
\end{equation} 
which is the Hodge decomposition~\cite{Eckmann1945}. One direct consequence of the Hodge decompostion~\eqref{eqn:hodge-decomp} is the following Poincar\'{e} inequalities,
\begin{align*}
&\| z_k \|_k \leq c \| \delta_k z \|_{k+1}, \quad z_k \in \operatorname{im}(\delta_k^*),\\
&\| z_k \|_k \leq c^* \| \delta_{k-1}^* z_k \|_{k-1}, \quad z_k \in \operatorname{im}(\delta_{k-1}).
\end{align*}
Here $c$ and $c^*$ are Poincar\'{e} constants which depend upon the topological structure of the graph.  Another consequence of the Hodge decomposition~\eqref{eqn:hodge-decomp} is that the Hodge Laplacian $L_k$ is positive semidefinite and the dimension of the null space is equal to the dimension of the homology $\operatorname{ker}(\delta_k)/\operatorname{im}(\delta_{k-1})$.  This implies the invertibility of the Hodge Laplacian $L_k$ on the orthogonal complement of the null space. 

In the next section, we will introduce a data-driven exterior calculus and establish analogues of these properties.

\section{The data-driven exterior calculus}\label{DDEC}

We now parameterize these graph calculus operators in a manner which allows recovery of traditional DEC schemes for PDE discretizations as a specific instance. We consider the de Rham complex as a prototypical means of analyzing structure-preserving differential operators in $\mathbb{R}^{d}$, and seek to develop its discrete approximation on a graph. In general the construction presented here may be used to approximate any exact sequence: we restrict our presentation to the de Rham complex as a single example to support later discretization of PDEs in $\mathbb{R}^d$. In three dimensions, the de Rham complex for $\Omega \in \mathbb{R}^3$ is given by
\begin{equation}
\begin{tikzcd}\label{eq:derham3d}
C^{\infty}(\Omega)  \arrow[transform canvas={yshift=0.4ex},"\operatorname{grad}"]{r}  & \arrow[transform canvas={yshift=-0.4ex},"{\operatorname{div}}^*"]{l}  \left[C^{\infty}(\Omega)\right]^3 \arrow[transform canvas={yshift=0.4ex},"\operatorname{curl}"]{r} & \arrow[transform canvas={yshift=-0.4ex},"{\operatorname{curl}}^*"]{l} \left[C^{\infty}(\Omega)\right]^3 \arrow[transform canvas={yshift=0.4ex},"\operatorname{div}"]{r}  & \arrow[transform canvas={yshift=-0.4ex},"{\operatorname{grad}}^*"]{l} C^{\infty}(\Omega)
\end{tikzcd}
\end{equation}
Here, coboundary operators $\operatorname{grad}$, $\operatorname{curl}$, and $\operatorname{div}$ satisfying $\operatorname{curl}\circ \operatorname{grad} = \operatorname{div} \circ \operatorname{curl} = 0$; codifferential operators $\operatorname{div}^*$, $\operatorname{curl}^*$, and $\operatorname{grad}^*$ satisfying $\operatorname{curl}^* \circ \operatorname{grad}^* = \operatorname{div}^* \circ \operatorname{curl}^* = 0$. 
To support later two-dimensional examples, we will also consider the restriction to two dimensions $\Omega \in \mathbb{R}^2$.
\begin{equation}
\begin{tikzcd}\label{eq:derham2d}
C^{\infty}(\Omega)  \arrow[transform canvas={yshift=0.4ex},"\operatorname{curl}"]{r}  & \arrow[transform canvas={yshift=-0.4ex},"{\operatorname{curl}}^*"]{l}  \left[C^{\infty}(\Omega)\right]^2 \arrow[transform canvas={yshift=0.4ex},"\operatorname{div}"]{r} & \arrow[transform canvas={yshift=-0.4ex},"{\operatorname{grad}}^*"]{l} C^{\infty}(\Omega)
\end{tikzcd}
\end{equation}
For completeness, we recall the alternate complex in $\mathbb{R}^2$
\begin{equation}
\begin{tikzcd}\label{eq:derham2dalt}
C^{\infty}(\Omega)  \arrow[transform canvas={yshift=0.4ex},"\operatorname{grad}"]{r}  & \arrow[transform canvas={yshift=-0.4ex},"{\operatorname{div}}^*"]{l}  \left[C^{\infty}(\Omega)\right]^2 \arrow[transform canvas={yshift=0.4ex},"\operatorname{curl}"]{r} & \arrow[transform canvas={yshift=-0.4ex},"{\operatorname{curl}}^*"]{l} C^{\infty}(\Omega)
\end{tikzcd}.
\end{equation}
In this work however, we consider only ~\eqref{eq:derham2d} as the natural complex to obtain conservation properties. 

\subsection{Data-driven coboundaries and codifferentials} \label{sec:coboudary-codifferential}
Consider the general case $(\Omega\in\mathbb{R}^d)$, we define new coboundary and codifferential operators by associating learnable metric information with the graph operators as follows,
\begin{equation} \label{eqn:discreteOps0}
\mathsf{d}_k := \mathbf{B}_{k+1} \delta_k \mathbf{B}_k^{-1}, \quad \text{and} \quad \mathsf{d}_k^* = \mathbf{D}_k^{-1} \delta_k^* \mathbf{D}_{k+1}.
\end{equation}
For example, in~$\mathbb{R}^3$, we have
\begin{equation}\label{eqn:discreteOps1}
GRAD = \mathbf{B}_1 \delta_0 \mathbf{B}_0^{-1}
\hspace{1cm}
CURL = \mathbf{B}_2 \delta_1 \mathbf{B}_1^{-1}
\hspace{1cm}
DIV = \mathbf{B}_3 \delta_2  \mathbf{B}_2^{-1}
\end{equation}
\begin{equation}\label{eqn:discreteOps2}
DIV^* = \mathbf{D}_0^{-1} \delta^*_0 \mathbf{D}_1
\hspace{1cm}
CURL^* =  \mathbf{D}_1^{-1} \delta^*_1 \mathbf{D}_2
\hspace{1cm}
GRAD^* =  \mathbf{D}_2^{-1} \delta^*_2 \mathbf{D}_3
\end{equation}
In~$\mathbb{R}^2$, we have
\begin{equation}\label{eqn:discreteOps2D1}
CURL = \mathbf{B}_1 \delta_0 \mathbf{B}_0^{-1}, \hspace{1cm} DIV = \mathbf{B}_2 \delta_1 \mathbf{B}_1^{-1}, 
\end{equation}
\begin{equation}\label{eqn:discreteOps2D2}
CURL^* = \mathbf{D}_0^{-1} \delta_0^* \mathbf{D}_1, \hspace{1cm} GRAD^* = \mathbf{D}_1^{-1} \delta_1^* \mathbf{D}_2, 
\end{equation}

Here, we denote discrete graph operations in capital letters, and use lower case for continuum counterparts. The $\mathbf{B}_k$ and $\mathbf{D}_k$ denote a diagonal tensor with positive entries weighting corresponding elements of $C^k$. They naturally provide chain maps as follows,

\begin{equation}\label{eqn:discreteDiagram}
\begin{tikzcd} 
C^0  \arrow{r}{\mathsf{d}_0} & C^1 \arrow{r}{\mathsf{d}_1} & C^2 \arrow{r}{\mathsf{d}_2}&  C^3 \arrow{r}{\mathsf{d}_3} & \cdots \arrow{r}{\mathsf{d}_{d-1}} & C^d\\
C^0  \arrow[shift left]{r}{\delta_0} \arrow[swap]{u}{\mathbf{B}_0}  & \arrow[shift left]{l}{\delta^*_1} C^1 \arrow[shift left]{r}{\delta_1} \arrow[swap]{u}{\mathbf{B}_1}  & \arrow[shift left]{l}{\delta^*_2} C^2 \arrow[shift left]{r}{\delta_2} \arrow[swap]{u}{\mathbf{B}_2}& \arrow[shift left]{l}{\delta^*_3} C^3 \arrow[shift left]{r}{\delta_3} \arrow[swap]{u}{\mathbf{B}_3}& \arrow[shift left]{l}{\delta^*_{d-1}} \cdots \arrow[shift left]{r}{\delta_{d-1}} & \arrow[shift left]{l}{\delta^*_d} C^d \arrow[swap]{u}{\mathbf{B}_d}\\
C^0 \arrow[swap]{u}{\mathbf{D}_0} & C^1 \arrow[swap]{u}{\mathbf{D}_1} \arrow{l}{\mathsf{d}_0^*} & C^2 \arrow[swap]{u}{\mathbf{D}_2} \arrow{l}{\mathsf{d}_1^*}&  C^3 \arrow[swap]{u}{\mathbf{D}_3} \arrow{l}{\mathsf{d}_2^*} &  \arrow{l}{\mathsf{d}_3^*} \cdots &   \arrow{l}{\mathsf{d}_{d-1}^*} C^d \arrow[swap]{u}{\mathbf{D}_d}
\end{tikzcd}
\end{equation}
Based on the definitions~\eqref{eqn:discreteOps1} and~\eqref{eqn:discreteOps2}, it is easy to verify that the diagram~\eqref{eqn:discreteDiagram} is commutative, i.e.,
\begin{equation*}
\mathbf{B}_{k+1} \delta_k = \mathsf{d}_k \mathbf{B}_k, \quad \text{and} \quad \mathbf{D}_k^{-1} \delta_k^* = \mathsf{d}_k^* \mathbf{D}_{k+1}^{-1}.
\end{equation*}
In $\mathbb{R}^3$, we have,
\begin{equation*}
\mathbf{B}_1 \delta_0 = GRAD \mathbf{B}_0, \quad \mathbf{B}_2 \delta_1 = CURL \mathbf{B}_1, \quad \mathbf{B}_3 \delta_2 = DIV \mathbf{B}_2,
\end{equation*}
and
\begin{equation*}
\delta_0^* \mathbf{D}_0^{-1} =  DIV^* \mathbf{D}_1^{-1}, \quad \mathbf{D}_1^{-1}\delta_1^*  =   CURL^* \mathbf{D}_2^{-1}, \quad \mathbf{D}_2^{-1} \delta_2^*  =  GRAD^*  \mathbf{D}_3^{-1}.
\end{equation*}
And in $\mathbb{R}^2$, similarly, we have
\begin{equation*}
\mathbf{B}_1 \delta_0 = CURL \mathbf{B}_0, \quad \mathbf{B}_2 \delta_1 = DIV \mathbf{B}_1,
\end{equation*}
and
\begin{equation*}
\delta_0^* \mathbf{D}_0^{-1} =  CURL^* \mathbf{D}_1^{-1}, \quad \mathbf{D}_1^{-1}\delta_1^*  =   GRAD^* \mathbf{D}_2^{-1}.
\end{equation*}

\begin{thm}\label{thm:exseq1}
The discrete derivatives $\mathsf{d}_k$ in~\eqref{eqn:discreteOps0} form an exact sequence if the simplicial complex is exact, and in particular $\mathsf{d}_{k+1} \circ \mathsf{d}_k = 0$.
\end{thm}
\begin{proof}
Because~\eqref{eqn:discreteDiagram} is a commutative diagram, the chain maps take cycles to cycles and boundaries to boundaries.  Therefore, if the simplicial complex is exact, then~\eqref{eqn:discreteOps1} forms an exact sequence. Moreover, 
$\mathsf{d}_{k+1} \circ \mathsf{d}_k = 0$ can be verified by the definitions~\eqref{eqn:discreteOps0}. 
\end{proof}

\begin{remark}
 In $\mathbb{R}^3$, we have~$CURL \circ GRAD  = DIV \circ CURL = 0$. And in $\mathbb{R}^2$, we have~$DIV \circ CURL = 0$.
\end{remark}

\begin{thm}\label{thm:exseq2}
The discrete derivatives $\mathsf{d}_k^*$ in~\eqref{eqn:discreteOps0} form an exact sequence if the simplicial complex is exact, and in particular $\mathsf{d}_k^* \circ \mathsf{d}_{k+1}^* = 0$.  
\end{thm}
\begin{proof}
The conclusion follows from the fact that~\eqref{eqn:discreteDiagram} is commutative and the definitions~\eqref{eqn:discreteOps0}. The proof is essentially the same as the proof of Theorem~\ref{thm:exseq1}
\end{proof}

\begin{remark}
In $\mathbb{R}^3$, we have~$DIV^* \circ CURL^* = CURL^* \circ GRAD^* = 0$. And in $\mathbb{R}^2$, we have~$CURL^* \circ GRAD^* = 0$
\end{remark}

We will refer to this collection of operators as a data-driven exterior calculus, with the understanding that the metric information encoded in $\mathbf{B}_k$ and $\mathbf{D}_k$ will be learned from data. Note that in the traditional low-order compatible/mimetic schemes, these metric tensors contain geometric information related to the oriented measures of mesh entities, such as cell volumes, face moments, etc. ~\cite{AdlerCavanaughHuZikatanov2020}.  Following from the exact sequence property, this exterior calculus structure inherits the following other properties of the graph calculus.

Naturally, based on the Hodge Laplacians~\eqref{def:hodge-laplacian} and~\eqref{eqn:discreteDiagram}, we can define generalized data-driven Hodge-Laplacians as follows,
\begin{equation*}
\Delta_k = \mathsf{d}_{k-1} \mathsf{d}_{k-1}^* + \mathsf{d}_k^* \mathsf{d}_k
\end{equation*}
For example, for practical applications, we consider the following Hodge-Laplacians in $\mathbb{R}^2$:
\begin{align*}
\Delta_1 &:= CURL \circ CURL^* + GRAD^* \circ DIV   \\
\Delta_2 &:= DIV \circ GRAD^*.
\end{align*}

The Hodge decomposition~\eqref{eqn:hodge-decomp} also can be generalized naturally.  Here, we choose $(\cdot, \cdot)_k :=(\cdot, \cdot)_{\mathbf{D}_k \mathbf{B}_k^{-1}}$ and denote its induced norm as $\| \cdot \|_k$.

\begin{thm}[Hodge Decomposition]
For $C^k$, the following decomposition holds
\begin{equation}\label{eqn:hodge-decomp-new}
C^k = \operatorname{im}( \mathsf{d}_{k-1} ) {\bigoplus}_k \operatorname{ker}(\Delta_k) {\bigoplus}_k \operatorname{im}(  \mathsf{d}_{k}^*),
\end{equation}
where ${\bigoplus}_k$ means the orthogonality with respect to the $(\cdot, \cdot)_{k}$-inner product.
\end{thm}
\begin{proof}
Since~\eqref{eqn:discreteDiagram} is a commutative diagram, following from the Hodge decomposition~\eqref{eqn:hodge-decomp}, $C^k$ can be naturally decomposed into three parts, $\operatorname{im}( \mathsf{d}_{k-1} )$, $\operatorname{im}( \mathsf{d}_{k}^* )$, and $\operatorname{ker}(\Delta_k) $.  Next we show this decomposition is orthogonal with respect to the $(\cdot, \cdot)_k$-inner product, i.e., $(\cdot, \cdot)_{\mathbf{D}_k \mathbf{B}_k^{-1}}$-inner product.  For $ \mathbf{u}_{k-1} \in C^{k-1} $ and $\mathbf{u}_{k+1} \in C^{k+1}$, we have
\begin{align*}
& \quad ( \mathsf{d}_{k-1} \mathbf{u}_{k-1}, \mathsf{d}_{k}^* \mathbf{u}_{k+1})_{k} \\
& = ( \mathbf{B}_{k} \delta_{k-1} \mathbf{B}_{k-1}^{-1} \mathbf{u}_{k-1}, \mathbf{D}_{k}^{-1} \delta_{k}^* \mathbf{D}_{k+1} \mathbf{u}_{k+1})_{\mathbf{D}_k \mathbf{B}_k^{-1}} \\
&=  ( \delta_{k-1} \mathbf{B}_{k-1}^{-1} \mathbf{u}_{k-1}, \delta_{k}^* \mathbf{D}_{k+1} \mathbf{u}_{k+1}) \\
& = ( \delta_k \delta_{k-1} \mathbf{B}_{k-1}^{-1} \mathbf{u}_{k-1}, \mathbf{D}_{k+1} \mathbf{u}_{k+1}) \\
& = 0
\end{align*}
For $\mathbf{h}_k \in \operatorname{ker}(\Delta_k)$, we have $\mathsf{d}_{k-1}^*  \mathbf{h}_k = 0$ and $\mathsf{d}_{k} \mathbf{h}_k = 0 $, which implies $\delta_{k-1}^* \mathbf{D}_k \mathbf{h}_k = 0$ and $\delta_k \mathbf{B}_k^{-1} \mathbf{h}_k = 0$. And then for $\mathbf{u}_{k-1} \in C^{k-1} $, 
\begin{align*}
& \quad ( \mathsf{d}_{k-1} \mathbf{u}_{k-1}, \mathbf{h}_k)_{k} \\
& = ( \mathbf{B}_{k} \delta_{k-1} \mathbf{B}_{k-1}^{-1} \mathbf{u}_{k-1}, \mathbf{h}_k)_{\mathbf{D}_k \mathbf{B}_k^{-1}}\\
& = ( \mathbf{D}_k  \delta_{k-1} \mathbf{B}_{k-1}^{-1} \mathbf{u}_{k-1},  \mathbf{h}_k) \\
& = (\mathbf{B}_{k-1}^{-1} \mathbf{u}_{k-1}, \delta_{k-1}^* \mathbf{D}_k \mathbf{h}_k ) \\
& = 0.
\end{align*}
On the other hand, for $\mathbf{u}_{k+1} \in C^{k+1}$, we have
\begin{align*}
& \quad (\mathsf{d}_{k}^* \mathbf{u}_{k+1}, \mathbf{h}_k)_{k} \\
& = (\mathbf{D}_{k}^{-1} \delta_{k}^* \mathbf{D}_{k+1} \mathbf{u}_{k+1}, \mathbf{h}_k)_{\mathbf{D}_k \mathbf{B}_k^{-1}} \\
& = (\mathbf{B}_{k}^{-1} \delta_{k}^* \mathbf{D}_{k+1} \mathbf{u}_{k+1}, \mathbf{h}_k) \\
& = (  \mathbf{D}_{k+1} \mathbf{u}_{k+1}, \delta_{k} \mathbf{B}_{k}^{-1} \mathbf{h}_k) \\
& = 0.
\end{align*}
Thus, the decomposition is orthogonal with respect to the $(\cdot, \cdot)_{k}$-inner product, which completes the proof.
\end{proof}

For example, in $\mathbb{R}^3$, we have the following Hodge decomposition when $k=1$,
\begin{equation*}
C^1 = \operatorname{im}(GRAD) {\bigoplus}_k \operatorname{ker}(\Delta_1) {\bigoplus}_k \operatorname{im}( CURL^*),
\end{equation*}
and when $k=2$
\begin{equation*}
C^2 = \operatorname{im}(CURL) {\bigoplus}_k \operatorname{ker}(\Delta_2) {\bigoplus}_k \operatorname{im}( GRAD^*).
\end{equation*}
In $\mathbb{R}^2$, we have the following Hodge decomposition when $k=1$
\begin{equation*}
C^1 = \operatorname{im}(CURL) {\bigoplus}_k \operatorname{ker}(\Delta_1) {\bigoplus}_k \operatorname{im}( GRAD^*).
\end{equation*}

Based on the Hodge decomposition, we have the following Poincar\'{e} inequality. 
\begin{thm}[Poincar\'{e} inequality]
For each $k$, there exists a constant $c_{P,k}$ such that
\begin{equation*}
\| \mathbf{z}_k \|_{k} \leq c_{P,k} \| \mathsf{d}_k \mathbf{z}_k \|_{k+1}, \quad \mathbf{z}_k \in \operatorname{im}(\mathsf{d}_{k}^*),
\end{equation*}
and another constant $c_{P,k}^*$ such that
\begin{equation*}
\| \mathbf{z}_{k} \|_{k} \leq c^*_{P,k} \| \mathsf{d}^*_{k-1} \mathbf{z}_{k} \|_{k-1}, \quad \mathbf{z}_{k} \in \operatorname{im}(\mathsf{d}_{k-1}).
\end{equation*}
Thus, for $\mathbf{u}_k \in C^k$, we have
\begin{equation*}
  \inf_{\mathbf{h}_k \in \operatorname{ker}(\Delta_k)}\| \mathbf{u}_k - \mathbf{h}_k \| _{k} \leq C \left( \| \mathsf{d}_k \mathbf{u}_k \|_{k+1} + \| \mathsf{d}^*_{k-1} \mathbf{u}_{k} \|_{k-1} \right),
\end{equation*}
where constant $C > 0$ only depends on $c_{P,k}$ and $c_{P,k}^*$.
\end{thm}
\begin{proof}
These inequalities are a direct consequence of the Hodge decomposition~\eqref{eqn:hodge-decomp-new}. 
\end{proof}

Since we are considering matrix representation, by direct computation, we can see that 
\begin{equation*}
c_{P,k} = \lambda_{\min} (\mathsf{d}_k^* \mathsf{d}_k)^{-1/2}.
\end{equation*}
where $\lambda_{\min}$ denotes the smallest non-trivial eigenvalue. For example, in $\mathbb{R}^3$, when $k=0$, 
\begin{equation*}
c_{P,0} = \lambda_{\min}(\mathsf{d}_0^* \mathsf{d}_0)^{-1/2} = \lambda_{\min}(\delta_0^* \mathbf{D}_{1} \mathbf{B}_{1} \delta_0)^{-1/2}=  \lambda_{\min}(DIV^* \circ GRAD)^{-1/2}.
\end{equation*}
Note that
\begin{equation*}
\min_i\{ (\mathbf{D}_{1} \mathbf{B}_{1})_{ii} \}  \lambda(\delta_0^* \delta_0) \leq \lambda(\delta_0^* \mathbf{D}_{1} \mathbf{B}_{1} \delta_0) \leq \max_i\{ (\mathbf{D}_{1} \mathbf{B}_{1})_{ii} \}  \lambda(\delta_0^* \delta_0).
\end{equation*}
This implies
\begin{equation*}
\min_i\{ (\mathbf{D}_{1} \mathbf{B}_{1})_{ii} \}  \lambda_{\min}(\delta_0^* \delta_0) \leq c_{P,0}^{-2} \leq \max_i\{ (\mathbf{D}_{1} \mathbf{B}_{1})_{ii} \}  \lambda_{\min}(\delta_0^* \delta_0),
\end{equation*}
which relates the Poincar\'{e} constant with $\lambda_{\min}(\delta_0^*\delta_0)$, also known as the Fielder value of the graph Laplacian~$L_0 = \delta_0^* \delta_0$. Classical works provide bounds upon the Fiedler eigenvalue in terms of the size and degree of a given graph, see for example \cite{chung1997spectral}.

Another consequence of the Hodge decomposition~\eqref{eqn:hodge-decomp-new} is the invertibility of the Hodge Laplacian $\Delta_k$ once we account for its kernel.

\begin{thm}[Invertibility of Hodge Laplacian] \label{thm:invert-Hodge-Laplacian}
The $k^{th}$-order Hodge Laplacian $\Delta_k$ is positive-semidefinite, with the dimension of its null-space equal to the dimension of the corresponding homology $H^k = \operatorname{ker}(\mathsf{d}_k)/ \operatorname{im}(\mathsf{d}_{k-1})$.
\end{thm}
\begin{proof}
For $\mathbf{u}_k \in C^k$, we have
\begin{align*}
 & \quad (\Delta_k \mathbf{u}_k, \mathbf{u}_k)_{k} \\
 & = ((\mathbf{B}_{k} \delta_{k-1} \mathbf{B}_{k-1}^{-1} \ \mathbf{D}_{k-1}^{-1} \delta_{k-1}^* \mathbf{D}_{k} +  \mathbf{D}_{k}^{-1} \delta_k^* \mathbf{D}_{k+1} \ \mathbf{B}_{k+1} \delta_k \mathbf{B}_{k}^{-1})\mathbf{u}_k, \mathbf{u}_k)_{\mathbf{D}_k \mathbf{B}_k^{-1}} \\
 & = ((\mathbf{D}_{k} \delta_{k-1} \mathbf{B}_{k-1}^{-1} \ \mathbf{D}_{k-1}^{-1} \delta_{k-1}^* \mathbf{D}_{k} +  \mathbf{B}_{k}^{-1} \delta_k^* \mathbf{D}_{k+1} \ \mathbf{B}_{k+1} \delta_k \mathbf{B}_{k}^{-1})\mathbf{u}_k, \mathbf{u}_k) \\
 & = ( \mathsf{d}_{k-1}^* \mathbf{u}_k,  \mathsf{d}_{k-1}^* \mathbf{u}_k)_{k-1} + (\mathsf{d}_k \mathbf{u}_k, \mathsf{d}_k \mathbf{u}_k)_{k+1}
\end{align*}
which shows that $\Delta_k$ is positive-semidefinite.  The second part follows directly from the Hodge decomposition~\eqref{eqn:hodge-decomp-new}.
\end{proof}

Theorem~\ref{thm:invert-Hodge-Laplacian} means that the following linear system of the Hodge Laplacian 
\begin{equation} \label{eqn:Hodge-Laplacian}
\Delta_k \mathbf{u}_k = \mathbf{f}_k,
\end{equation}
is solvable as long as $\mathbf{f}_k \in C^k$ is orthogonal to $\operatorname{ker}(\Delta_k)$ with respect to the $(\cdot, \cdot)_{k}$-inner product.

\subsection{Nonlinear Perturbation of Hodge-Laplacian Problems} \label{sec:nonlinear-pertubation}
In many cases, it is helpful to consider the mixed form of the Hodge-Laplacian problem~\eqref{eqn:Hodge-Laplacian} as it naturally provides connections to integral balance laws and conservation principles~\cite{ArnoldFalkWinther2006b,ArnoldFalkWinther2010}.  To this end, let us introduce a new variable 
$\mathbf{w}_{k-1} : = \mathsf{d}_{k-1}^* \mathbf{u}_k$ and the mixed form of \eqref{eqn:Hodge-Laplacian} as follows,
\begin{align}
\mathbf{w}_{k-1} - \mathsf{d}_{k-1}^* \mathbf{u}_{k} &= 0, \label{eqn:mixed-Hodge-Laplacian-1}\\
\mathsf{d}_{k-1}  \mathbf{w}_{k-1}  +  \mathsf{d}_k^* \ \mathsf{d}_k \mathbf{u}_k &= \mathbf{f}_k. \label{eqn:mixed-Hodge-Laplacian-2}
\end{align}
This class of problems describes several canonical second-order elliptic operators; for example, in $\mathbb{R}^2$, for $k=2$ we obtain the Darcy flow model
\begin{center}
\begin{tabular}{c c c}
$\mathbf{F} + \kappa \nabla \phi = 0$ & $\rightarrow$ & $\mathbf{w}_1 - GRAD^* \mathbf{u}_0 = 0$\\
    $\nabla \cdot \mathbf{F} = f$ & & $DIV \mathbf{w}_1 = \mathbf{f}_0,$
\end{tabular}
\end{center}
and for $k=1$ we obtain the magnetostatics model, after applying a vector potential for the magnetic field and applying a suitable gauge condition \cite{BochevHuEtAl2008_AlgebraicMultigridApproachBased,bossavit2001stiff}.
\begin{center}
\begin{tabular}{c c c}
$\nabla \times \mathbf{J} = \mathbf{f}$ & $\rightarrow$ & $\mathbf{w}_0 - CURL^* \mathbf{u}_1 = 0$,\\
$\nabla \cdot \mathbf{B} = 0$ & & $ CURL \mathbf{w}_0 + GRAD^* \circ DIV \mathbf{u}_1 = \mathbf{f}_1.$\\
$\mathbf{J} = \mu \mathbf{B}$ & &
\end{tabular}
\end{center}



While this model form is appropriate for learning, e.g. diffusion coefficients corresponding to second-order elliptic problems, realistic problems require accounting for nonlinearities. With this in mind, we introduce a nonlinear perturbation of the fluxes while leaving the relevant conservation statements untouched. This preserves the exterior calculus structure while incorporating data into fluxes only, which are traditionally more empirical. Any parameterization may be used for the nonlinearities, but we consider deep neural networks. As a result, we obtain a nonlinear perturbation of a Hodge-Laplacian problem in the mixed form as follows,
\begin{align}
&\mathbf{w}_{k-1} = \mathsf{d}^*_{k-1} \mathbf{u}_{k} + \epsilon \mathcal{NN}(\mathsf{d}^*_{k-1} \mathbf{u}_{k}; \xi), \label{eqn:perturb-mixed-1}\\
&\mathsf{d}_{k-1} \mathbf{w}_{k-1}   +  \mathsf{d}_k^* \ \mathsf{d}_k \mathbf{u}_k = \mathbf{f}_k. 
\label{eqn:perturb-mixed-2}
\end{align}
The corresponding primal form is
\begin{equation}\label{eqn:modelForm}
\Delta_k \mathbf{u}_k + \epsilon \mathsf{d}_{k-1} \circ \mathcal{NN}(\mathsf{d}^*_{k-1} \mathbf{u}_{k}; \xi) = \mathbf{f}_k
\end{equation}
Later in this section, we will theoretically show that when $\epsilon>0$ is sufficiently small, the nonlinear problem~\eqref{eqn:perturb-mixed-1} and \eqref{eqn:perturb-mixed-2} remains well-posed.  First, let us look at some examples. In $\mathbb{R}^2$, when $k=2$, we have
\begin{align*}
DIV \circ GRAD^* \mathbf{u}_0 + \epsilon DIV \circ \mathcal{NN}(GRAD^* \mathbf{u}_{0}; \xi) = \mathbf{f}_0 \\
= \Delta_0 \mathbf{u}_2 + \epsilon DIV \circ \mathcal{NN}(GRAD^* \mathbf{u}_{0}; \xi) = \mathbf{f}_0 
\end{align*}
and when $k=1$, we have
\begin{align*}
GRAD^* \circ DIV \mathbf{u}_1 + CURL \circ CURL^* \mathbf{u}_1 + \epsilon CURL \circ \mathcal{NN}(CURL^* \mathbf{u}_{1}; \xi) = \mathbf{f}_1,\\
= \Delta_1 \mathbf{u}_1 + \epsilon CURL \circ \mathcal{NN}(CURL^* \mathbf{u}_{1}; \xi) = \mathbf{f}_1.
\end{align*}

\subsection{Well-posedness} \label{sec:well-posedness}
Next we investigate the well-posedness of the perturbed Hodge-Laplacian problem~\eqref{eqn:perturb-mixed-1}-\eqref{eqn:perturb-mixed-2}.  We write the perturbed problem in the primal form, i.e.,
\begin{equation} \label{eqn:pertub-primal-Hodge}
\Delta_k \mathbf{u}_k + N(\mathbf{u}_k) = \mathbf{f},
\end{equation}
where $N(\mathbf{u}_k):= \epsilon \mathsf{d}_{k-1} \circ \mathcal{NN}(\mathsf{d}^*_{k-1} \mathbf{u}_k; \xi)$.  Consider the space  $\mathbb{V} = C^k \backslash \operatorname{ker}(\Delta_k)$, we introduce the following weak formulation of~\eqref{eqn:pertub-primal-Hodge}, 
\begin{equation}\label{eqn:pertub-mixed-Hodge-weak}
a(\mathbf{u}_k, \mathbf{v}) + \epsilon (\mathcal{NN}(\mathsf{d}_{k-1}^*\mathbf{u}_k), \mathsf{d}_{k-1}^* \mathbf{v})_{k-1} = (\mathbf{f}, \mathbf{v})_k, \quad \mathbf{v} \in \mathbb{V},
\end{equation}
where
\begin{equation*}
a(\mathbf{u}, \mathbf{v}) 
:= ( \mathsf{d}^*_{k-1} \mathbf{u}, \mathsf{d}^*_{k-1} \mathbf{v} )_{k-1}+  ( \mathsf{d}_k \mathbf{u}, \mathsf{d}_k \mathbf{v})_{k+1}, \quad \forall \, \mathbf{u}, \, \mathbf{v} \in \mathbb{V},
\end{equation*}
and its induced norm is $\| \mathbf{u} \|_a := \sqrt{a(\mathbf{u}, \mathbf{u})} = \sqrt{ \| \mathsf{d}_{k-1}^* \mathbf{u}\|_{k-1} + \| \mathsf{d}_k \mathbf{u} \|_{k+1}^2 }$.  We assume Liptschitz continuity of the nonlinear perturbation, i.e. that there exists a constant $L_N > 0$, such that,
\begin{equation}\label{assump:Lip}
\| \mathcal{NN}(\mathbf{v}_{k-1}) - \mathcal{NN}(\mathbf{w}_{k-1}) \|_{k-1} \leq L_N \| \mathbf{v}_{k-1} - \mathbf{w}_{k-1} \|_{k-1}, \quad \forall \mathbf{v}_{k-1}, \, \mathbf{w}_{k-1} \in C^{k-1}.
\end{equation}
In addition, we also assume that
\begin{equation} \label{assump:zero-flux}
\mathcal{NN}(\mathbf{0}) = \mathbf{0},
\end{equation}
which means that the nonlinear perturbation reduces to a linear gradient closure in the limit as the gradient becomes small.

The main tool we use is the Leray-Schauder fixed point theorem~\cite{gilbarg2015elliptic} 
We define $T: \mathbb{V} \mapsto \mathbb{V}$ such that for each $\mathbf{w} \in \mathbb{V}$, $\widetilde{\mathbf{u}} : = T(\mathbf{w}) \in \mathbb{V}$ is given as the solution of the following linear problem,
\begin{equation}
a(\widetilde{\mathbf{u}}, \mathbf{v}) + \epsilon(\mathcal{NN}( \mathsf{d}_{k-1}^*\mathbf{w}), \mathsf{d}_{k-1}^* \mathbf{v})_{k-1} = (\mathbf{f},\mathbf{v})_k, \quad \forall \mathbf{v} \in \mathbb{V}.
\end{equation}
The map $T$ is clearly continuous and, therefore, compact in the finite dimensional space $\mathbb{V}$.  The sovability of~\eqref{eqn:pertub-mixed-Hodge-weak} is thus equivalent to the solvability of the equation $\mathbf{u}_k = T(\mathbf{u}_k)$ in $\mathbb{V}$, which is a fixed point problem.

If $\lambda > 0$ and $\mathbf{w}$ satisfies $T(\mathbf{w}) = \lambda \mathbf{w}$, then
\begin{equation*}
\lambda a(\mathbf{w},\mathbf{v}) + \epsilon(\mathcal{NN}(\mathsf{d}_{k-1}^*\mathbf{w}),\mathsf{d}_{k-1}^*\mathbf{v})_{k-1} = (\mathbf{f}, \mathbf{v})_k, \quad \forall \mathbf{v} \in \mathbb{V}.
\end{equation*}
By choosing $\mathbf{v} = \mathbf{w}$, we obtain
$$
\lambda \| \mathbf{w} \|_a^2 \leq \| \mathbf{f} \|_{-a} \| \mathbf{w} \|_a + \epsilon \| \mathcal{NN}(\mathsf{d}_{k-1}^* \mathbf{w}) \|_{k-1} \| \mathsf{d}_{k-1}^* \mathbf{w} \|_{k-1} 
$$
By the Liptschitz continuity assumption~\eqref{assump:Lip} and assumption~\eqref{assump:zero-flux}, we have
\begin{align*}
\| \mathcal{NN}(\mathsf{d}_{k-1}^* \mathbf{w}) \|_{k-1} = \| \mathcal{NN}(\mathsf{d}_{k-1}^* \mathbf{w}) - \mathcal{NN}(\mathbf{0}) \|_{k-1} \leq L_N \| \mathsf{d}_{k-1}^* \mathbf{w}\|_{k-1}.
\end{align*}
Using the fact that $\| \mathsf{d}_{k-1}^* \mathbf{w} \|_{k-1} \leq \| \mathbf{w} \|_a$, we have
$$
\lambda \| \mathbf{w} \|_a^2 \leq \left( \| \mathbf{f}\|_{-a} + \epsilon L_N \| \mathbf{w} \|_a \right) \| \mathbf{w} \|_a,
$$
therefore,
$$
\lambda \leq  \frac{\| \mathbf{f}\|_{-a} + \epsilon L_N \| \mathbf{w} \|_a}{\|\mathbf{w} \|_a}.
$$
Thus, $\lambda < 1$ holds true for any $\mathbf{w}$ being on the boundary of the ball in $\mathbb{V}$ centered at the origin with radius 
$$
\rho \geq \frac{\|\mathbf{f} \|_{-a}}{(1 -  \epsilon L_N)}.
$$ 
Consequently, the Leray-Schauder fixed point theorem implies that the nonlinear map $T$ has a fixed point in any ball centered at the origin with radius $\rho \geq \frac{\|\mathbf{f} \|_{-a}}{(1 - \epsilon L_N)}$. This fixed point is a solution of the equation~\eqref{eqn:pertub-mixed-Hodge-weak}.  

\begin{thm}
 Assume~\eqref{assump:Lip} and~\eqref{assump:zero-flux} hold. If $\epsilon L_N<1$, the equation~\eqref{eqn:pertub-mixed-Hodge-weak} has an unique solution $\mathbf{u}_k \in \mathbb{V}$ satisfies 
	\begin{equation} \label{eqn:solu-bound}
	\| \mathbf{u}_k \|_a \leq \frac{\| \mathbf{f} \|_{-a}}{(1 - \epsilon L_N)}.
	\end{equation}
\end{thm}
\begin{proof}
	The existence has been discussed before the theorem. Let $\mathbf{u}_k$ be the solution and $\mathbf{v}=\mathbf{u}_k$ in~\eqref{eqn:pertub-mixed-Hodge-weak}, we have 
	$$
	 \| \mathbf{u}_k \|_a^2 \leq \left( \| \mathbf{f}\|_{-a} + \epsilon L_N \| \mathbf{u}_k \|_a \right) \| \mathbf{u}_k \|_a,
	$$
	which implies~\eqref{eqn:solu-bound}. 

Next we prove uniqueness of ~\eqref{eqn:pertub-mixed-Hodge-weak} under the same assumptions.  Let $\mathbf{u}_k$ and $\bar{\mathbf{u}}_k$ be two solutions of~\eqref{eqn:pertub-mixed-Hodge-weak}.  Denoting $\mathbf{e} = \mathbf{u}_k - \bar{\mathbf{u}}_k$, we have,
$$
a(\mathbf{e},\mathbf{v}) + \epsilon(\mathcal{NN}( \mathsf{d}_{k-1}^* \mathbf{u}_k), \mathsf{d}_{k-1}^*\mathbf{v})_{k-1} - \epsilon (\mathcal{NN}(\mathsf{d}_{k-1}^*\bar{\mathbf{u}}_k),\mathsf{d}_{k-1}^*\mathbf{v})_{k-1} = 0.
$$
Letting $\mathbf{v} = \mathbf{e}$ and assuming Lipschitz continuity,
\begin{equation}
\| \mathcal{NN}(\mathbf{v}_{k-1}) - \mathcal{NN}(\mathbf{w}_{k-1}) \|_{k-1} \leq L_N \| \mathbf{v}_{k-1} - \mathbf{w}_{k-1} \|_{k-1}, \quad \forall \mathbf{v}_{k-1}, \, \mathbf{w}_{k-1} \in C^{k-1},
\end{equation}
we arrive at,
\begin{align*}
\| \mathbf{e} \|_a^2 \leq \epsilon L_N \| \mathsf{d}_{k-1}^* \mathbf{e} \|_{k-1}^2 \leq \epsilon L_N \| \mathbf{e} \|^2_a
\end{align*}
which implies the uniqueness of the solution since $\epsilon L_N < 1$.  This completes the proof.
\end{proof}
We finally consider design of an architecture which satisfies this condition to ensure the extracted model is solvable, considering multilayer perceptrons as canonical architectures \cite{rosenblatt1961principles}.
\begin{remark}
Consider a $L$-layer neural net which has the following structure
\begin{equation*}
\mathcal{NN}(\mathbf{x}) := \psi_{L} \circ T_L \circ \cdots \circ \psi_1 \circ T_1(\mathbf{x})
\end{equation*}
where $T_{\ell}(\mathbf{x}_{\ell}) : = M_{\ell}\mathbf{x}_{\ell} + \mathbf{b}_{\ell}$ is an affine function and $\psi_{\ell}$ is a nonlinear activation function.  If we assume $\psi_{\ell}$ are $1$-Lipschitz nonlinear functions (e.g., ReLU, Leaky ReLU, tanh, sigmoid) and define $\mathcal{N}_{\ell} : = T_{\ell} \circ \psi_{\ell-1} \cdots \circ \psi_1 \circ T_1(\mathbf{x})$, then the Lipschitz constant $L_N$~\eqref{assump:Lip} of $\mathcal{NN}$ can be estimated as follows,
\begin{align*}
L_N &= \sup_{\mathbf{x} \in \mathbb{R}^n} \| \operatorname{diag}(\psi_L'(\mathcal{N}_L))M_L \cdots M_2 \operatorname{diag}(\psi_1'(\mathcal{N}_1))M_1 \|_{k-1} \\
& \leq \max_{i}\{(\mathbf{D}_{k-1}\mathbf{B}_{k-1}^{-1})_{ii}\} \sup_{\mathbf{x} \in \mathbb{R}^n} \| \operatorname{diag}(\psi_L'(\mathcal{N}_L))M_L \cdots M_2 \operatorname{diag}(\psi_1'(\mathcal{N}_1))M_1 \| \\
& \leq \max_{i}\{(\mathbf{D}_{k-1}\mathbf{B}_{k-1}^{-1})_{ii}\} \prod_{\ell=1}^{L} \| M_\ell \|
\end{align*}
Therefore, if we choose $0 < \epsilon < \left( \max_{i}\{(\mathbf{D}_{k-1}\mathbf{B}_{k-1}^{-1})_{ii}\} \prod_{\ell=1}^{L} \| M_\ell \| \right)^{-1}$, then the assumption $\epsilon L_n < 1$ would be satisfied. We also note that a more accurate upper bound of the Lipschitz constant $L_N$ can be computed numerically using the advanced algorithms developed in ~\cite{ScamanVirmaux2019}, providing a tighter bound on $\epsilon$. 
\end{remark}

\subsection{Construction of chain complex}

The model introduced in the previous section assumes access to an underlying graph to apply the DDEC to. Motivated by our surrogate modeling application, we assume access to a very fine polygonal mesh and access to a high-fidelity PDE solution defined as oriented moments of mesh entities (i.e. cell average scalar potentials, face average fluxes, edge average circulations). We next show how a coarsening of the underlying fine mesh preserves some structure allowing a particularly simple implementation; essentially, given graph grad/curl/div matrices on the fine mesh, we derive coarsening matrices that encode proper orientations for the coarsened complex and the desired $\mathsf{d}_k$ and $\mathsf{d}^*_k$ operators. We consider here graph-cut coarsening available in packages such as METIS \cite{karypis1997metis}, which partitions the domain $\Omega$ into $N_c$ disjoint volumetric subdomains (or $d$-cells), and derive appropriately oriented lower degree $k$-cells. In Figure \ref{fig:coarseningCartoon} we provide a cartoon of the process. While the following presentation provides a mathematical description of the process via a commutative diagram, it practically will provide a simple implementation of the coarse coboundary operators $\delta_k$ in terms of a few simple matrices, greatly simplifying implementation.

\begin{remark} This is one particular construction appropriate for synthetic data which takes advantage of available adjacency matrices of an underlying fine mesh. We stress however that one may apply the calculus to \textit{any} appropriately defined graph. For example, in experimental contexts the bins associated with histograms may be used instead, or the calculus may be applied to graphs with no associated partition of space at all (see e.g. \cite{jiang2011statistical}). We further comment that the graph-cut coarsening assumed here provides a quasi-uniform partition of space that does not take advantage of the data; in another work we consider spectral graph partitioning strategies to obtain data-driven partitions tuned to give optimal representations of data.
\end{remark}

\begin{figure}[!h]\label{fig:coarseningCartoon}
\centering
$
\vcenter{\hbox{\begin{tikzpicture}
\node[circle, fill,inner sep=1.5pt] (2) at (-1,-1) {};
\node[circle, fill,inner sep=1.5pt] (0) at (-2,0) {};
\node[circle, fill,inner sep=1.5pt] (1) at (-1.5,1.2) {};
\node[circle, fill,inner sep=1.5pt] (4) at (1,1) {};
\node[circle, fill,inner sep=1.5pt] (5) at (1.3,-0.75) {};
\node[circle, fill,inner sep=1.5pt] (3) at (0,0) {};

\draw[black] (0) -- (1);
\draw[black] (0) -- (2);
\draw[black] (0) -- (3);
\draw[black] (1) -- (3);
\draw[black] (1) -- (4);
\draw[black] (2) -- (3);
\draw[black] (2) -- (5);
\draw[black] (3) -- (4);
\draw[black] (3) -- (5);
\draw[black] (4) -- (5);

\begin{pgfonlayer}{background}
\fill[red!15] (0.center) -- (1.center) -- (3.center);
\fill[red!15] (0.center) -- (2.center) -- (3.center);
\fill[blue!15] (1.center) -- (3.center) -- (4.center);
\fill[blue!15] (2.center) -- (3.center) -- (5.center);
\fill[blue!15] (3.center) -- (4.center) -- (5.center);
\end{pgfonlayer}
\end{tikzpicture}}}
\vcenter{\hbox{\hskip 0.6cm $\longrightarrow$\hskip 0.6cm }}
\vcenter{
	\hbox{\begin{tikzpicture}
\node[circle, fill,inner sep=1.5pt] (2) at (-1  ,   -1) {};
\node[circle, fill,inner sep=0.0pt] (0) at (-2  ,    0) {};
\node[circle, fill,inner sep=1.5pt] (1) at (-1.5,  1.2) {};
\node[circle, fill,inner sep=0.0pt] (4) at (1   ,    1) {};
\node[circle, fill,inner sep=0.0pt] (5) at (1.3 ,-0.75) {};
\node[circle, fill,inner sep=0.0pt] (3) at (0   ,    0) {};

\draw[black] (0) -- (1);
\draw[black] (0) -- (2);
\draw[black] (1) -- (3);
\draw[black] (1) -- (4);
\draw[black] (2) -- (3);
\draw[black] (2) -- (5);
\draw[black] (4) -- (5);

\begin{pgfonlayer}{background}
\fill[red!15] (0.center) -- (1.center) -- (3.center);
\fill[red!15] (0.center) -- (2.center) -- (3.center);
\fill[blue!15] (1.center) -- (3.center) -- (4.center);
\fill[blue!15] (2.center) -- (3.center) -- (5.center);
\fill[blue!15] (3.center) -- (4.center) -- (5.center);
\end{pgfonlayer}

\end{tikzpicture}}}
$\\
\vskip 12pt

\raisebox{-.5\height}{\includegraphics[width=0.4\linewidth]{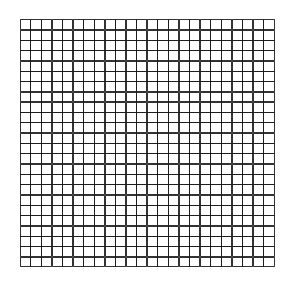}}
$\vcenter{\hbox{\hskip 0.6cm $\longrightarrow$\hskip 0.6cm }}$
\raisebox{-.5\height}{\includegraphics[width=0.4\linewidth]{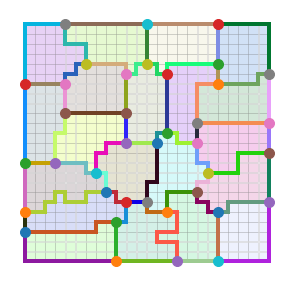}}

\caption{The coarsening process takes a fine mesh discretizing $\Omega$ in $d$-cells, partitions them into disjoint subdomains and uses the fine scale mesh to derive coarsened $k$-cells and coboundary operators $\partial_k$ with consistent orientations. Values on the coarsened $k$-cells correspond to sums of the constituent fine-grained moments with proper accounting of orientation. \textit{Top:} five $2$-cells coarsened into a pair of $2$-cells. \textit{Bottom:} a representative coarsening of a Cartesian mesh. Coarsened $0$-, $1$-, and $2$-cells are each colored differently and are a formal linear sum of elements of the fine mesh.}
\end{figure}
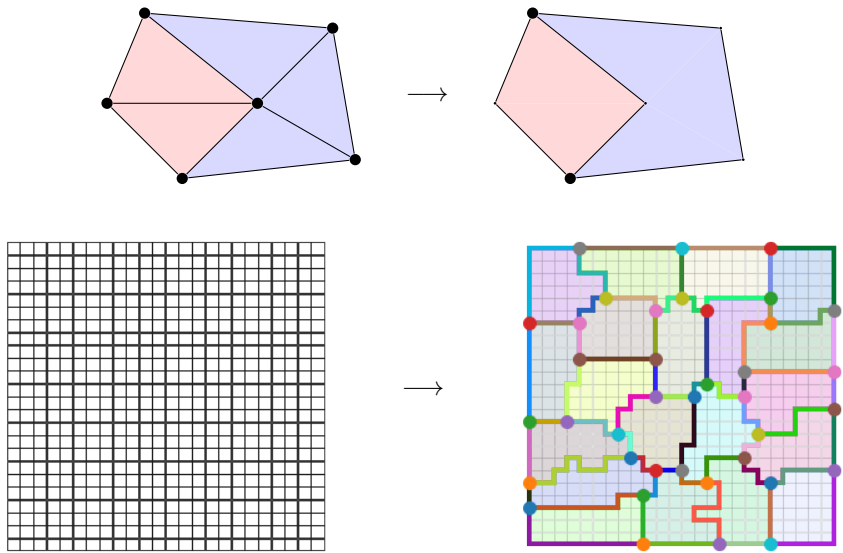



We assume the high-fidelity solution is associated with a $d$-dimensional chain complex, with $k$-cells denoted by $F_k$ for $k<d$ and $\partial_k^{fine}:F_{k+1}\rightarrow F_k$ as defined in Section \ref{sec:graphcalc}, and coboundary $\delta_k^{fine}:F^k\rightarrow F^{k+1}$ encoded via the adjacency matrices generally available in mesh data structures. Here, we use the symbol $F$ and superscript $fine$ to denote fine scale objects. Our goal is to construct a coarse complex, denoted by the symbol $C$, that can be used in the DDEC framework as we discussed in Section~\ref{sec:coboudary-codifferential}-\ref{sec:well-posedness}.  The construction contains two steps. First, we will coarsen the fine level set of the $k$-chain, $F_k$, to obtain the coarse level set of the $k$-chain, $C_k$. This can be done inductively by starting with coarsening the $d$-cells. Then we will define the coarse boundary and coboundary operators via relating the fine and coarse level complex properly. 

 
\paragraph{Constructing $C_k$ and $C^k$.} To inductively define $C_k$ from $F_k$, we start with coarsening the $d$-cells. Given a partitioning of the $N_f$ $d$-cells into $N_c$ disjoint, connected subsets, denoted by $F_d=\sqcup_{i=1}^{N_c} P_i$.
we define the space of coarse $d$-chain $C_d$ as a subspace of $F_d$ based on the partition $\{ P_i \}_{i=1}^{N_c}$ with the following natural inclusion
\begin{equation*}
\iota_d \in \mathbb{R}^{N_f \times N_c}, \quad (\iota_d)_{ij} =
\begin{cases}
1, \qquad f_i \in P_j, \\
0, \qquad f_i \notin P_j.
\end{cases}
\end{equation*}
This has a dual $C^d$ identified as a subspace of $F^d$. 

Now, consider the image of $C_d$ under $\partial_{d-1}^{fine}$, consisting of the fine $d$-chains lying at partition interfaces. We define the element $c_{ij}:= P_i \cap P_j$ (oriented consistently) for $i,j=1, 2,\ldots N_c$ and their span as the the coarse $(d-1)$ chains $C_{d-1}$.  Then we can proceed inductively.  With the chain and cochain spaces indexed by $d-k$, consider the image of $C_{d-k}$ under the map $\partial_{d-k-1}^{fine}$. This consists of the $(d-k-1)$-interfaces between the coarse $d-k$-chains. Define the element $$c_{i_1i_2 \cdots i_{d-k-1}}=P_{i_1}\cap P_{i_2}\cap \cdots \cap P_{i_{d-k-1}}$$ The span of these elements is the space of coarse $(d-k-1)$-chains $C_{d-k-1}$. As before, by identifying $C_{d-k-1}$ as a subset of $F_{d-k-1}$, we can define the dual cochains $C^{d-k-1}$. 

\begin{remark}
It may be the case that $c_{i_1 i_2 \cdots  i_{d-k-1}}$ is disjoint. As a choice of implementation detail, each of the connected components can be taken as a separate chain, or their disjoint union can be taken as a single chain. What results from this choice is the treatment of parallel fluxes between partitions as either: a sum of individual distinct fluxes between the partitions, or the sum of those fluxes as a single effective quantity between the partitions. We choose the former in our computational examples.
\end{remark}

By construction, we have $C_k \subset F_k$, for $k=0,\cdots, d$. Therefore, there is a natural inclusion $\iota_k: C_k \mapsto F_k$ whose entries are $1$, $-1$, or $0$ depending on the partition and orientation. Furthermore, we also have $C^k \subset F^k$ and the corresponding natural inclusions are defined as $\iota^k := \iota_{d-k}$, for $k=0, \cdots, d$.


\paragraph{Building boundary and coboundary operators.} 
In order to define the coarse boundary and coboundary operators that connects $C_k$ and $C^k$, we first define the linear projections $\pi_k: F_k \mapsto C_k$ as 
$$
\pi_k = \left(\iota_k^T \iota_k\right)^{-1} \iota_k^T,
$$ 
and $\pi^k: F^k \mapsto C^k$ as 
$$
\pi^k = \left(\left(\iota^k\right)^T \iota^k\right)^{-1} \left(\iota^k\right)^T.
$$ 
Note that, $\pi^k = \pi_{d-k}$ since $\iota^k = \iota_{d-k}$.

With these in hand, we define the coarse boundary operators as follows,
\begin{eqnarray*}
\partial_{k} :=& \pi_{k} \cdot \partial_{k}^{fine} \cdot \iota_{k+1},
\end{eqnarray*}
and the coarse coboundary operators 
\begin{eqnarray*}
\delta_{k} :=& \pi^{k+1} \cdot \delta_{k}^{fine} \cdot \iota^{k}.
\end{eqnarray*}

\begin{remark}
The linear projections $\pi_k$ and $\pi^k$ are least-squares projections. Since the natural inclusions $\iota_k$ and $\iota^k$ are constructed based on partitions, $\iota_k^T \iota_k$ and $(\iota^k)^T \iota^k$ are diagonal matrices with diagonal entries equal to the number of $k$-chains in the corresponding partition. Thus, in a practical implementation, we can simply use $\pi_k = \iota_k^T$ and $\pi^k = (\iota^k)^T$ and, based on the definition~\eqref{eqn:discreteOps0}, $(\iota_k^T \iota_k)^{-1}$ and $((\iota^k)^T \iota^k)^{-1}$ are implicitly absorbed in $\mathbf{B}_k$ and $\mathbf{D}_k$. 
\end{remark}



To summarize, in the case of $d=2$, all of the spaces discussed are related through the following commutative diagram:
\begin{equation}\label{eqn:coarseningComplex}
\begin{tikzcd}[row sep=1.6em,column sep=1.6em]
0 \arrow[rr] & & F^0 \arrow[rr, "\delta_0^{fine}"] \arrow[dr,shift right, swap,"\pi^0"] \arrow[dd,swap,"", near start] && F^1 \arrow[rr,"\delta_1^{fine}"] \arrow[dd,swap,"" near start] \arrow[dr,shift right, swap,"\pi^1"] && F^2 \arrow[dd,swap,"" near start] \arrow[dr,shift right, swap,"\pi^2"] \arrow[rr] & & 0\\
& 0 \arrow[rr,, crossing over] & & C^0 \arrow[rr,crossing over,"\delta_0" near start] \arrow[ul,shift right,swap,->,"\iota^0"] && C^1 \arrow[rr,crossing over,"\delta_1" near start] \arrow[ul,shift right,swap,->,"\iota^1"] && C^2 \arrow[rr] \arrow[ul,shift right,swap,->,"\iota^2"] & & 0 \\
0 & & F_0 \arrow[ll] \arrow[dr,shift right, swap,"\pi_0"] && F_1 \arrow[ll,"\partial_0^{fine}" near start] \arrow[dr,shift right, swap,"\pi_1"] && F_2 \arrow[dr,shift right, swap,"\pi_2"] \arrow[ll,"\partial_1^{fine}" near start] & & \arrow[ll] 0\\
& 0 & & C_0 \arrow[ll] \arrow[uu,<-,crossing over,"" near end] \arrow[ul,shift right,swap,->,"\iota_0"] && C_1 \arrow[ll,"\partial_0"] \arrow[uu,<-,crossing over,"" near end] \arrow[ul,shift right,swap,->,"\iota_1"] && C_2 \arrow[ll,"\partial_1"] \arrow[uu,<-,crossing over,"" near end] \arrow[ul,shift right,swap,->,"\iota_2"] & & \arrow[ll] 0
\end{tikzcd}
\end{equation}
The top-front row of this diagram corresponds to the middle row of the chain complex \eqref{eqn:discreteDiagram}. The DDEC complex results from perturbing these operators as in \eqref{eqn:discreteOps0}.

\section{PDE-constrained optimization}

We finally turn toward the question of how to fit a model of the form~\eqref{eqn:modelForm} to data. We assume access to data via the coarsening process of the previous section, and denote by $\mathbf{u}_{data}$ a vector concatenating the coarsened $\mathbf{w}_{k+1}$ and $\mathbf{u}_k$ degrees of freedom. We may have only partial data: observations of possibly a single field $\mathbf{w}_{k+1}$ or $\mathbf{u}_k$, or observations on a subset of the chain complex. We then concisely express the boundary value problem in~\eqref{eqn:modelForm} via the nonlinear operator $\mathcal{L}_\xi [\mathbf{u}] = \mathbf{f}$, where we have lumped all model parameters into $\xi$ (i.e. $\mathbf{B}_k$, $\mathbf{D}_k$, and neural network weights and biases), and refer to this as the \textit{forward problem}. We postpone a problem-specific discussion of how boundary conditions will be imposed for the following section, and assume the forward problem is prescribed such that BCs are imposed naturally. We seek a solution to the following quadratic program with nonlinear equality constraints

\begin{align}\label{eq:equalityQP}
    \underset{\xi}{\text{argmin}} ||\mathbf{u} - \mathbf{u}_{data}||^2\\
    \text{such that } \mathcal{L}_\xi[\mathbf{u}] = \mathbf{f},
\end{align}
where $||\cdot||$ denotes the $\ell_2$ norm. 

To avoid handling the equality constraint, one may introduce a single scalar penalty parameter $\lambda$ and use a gradient descent optimizer to solve
\begin{align}
    \underset{\xi}{\text{argmin}} ||\mathbf{u} - \mathbf{u}_{data}||^2 + \lambda ||\mathcal{L}_\xi[\mathbf{u}] - \mathbf{f}||^2.
\end{align}
This approach resembles currently popular approaches such as physics-informed neural networks and is simple to implement in machine learning libraries using automatic differentiation to implement first-order optimization schemes. However, $\lambda$ becomes a hyperparameter introducing well-known issues with training, and ultimately results in $\mathcal{L}_\xi[\mathbf{u}]=\mathbf{f}$ holding only to within optimization error. We instead enforce the equality constraint exactly, introducing Lagrange multipliers $\mathbf{\lambda}$ as follows.

\begin{align}\label{eq:equalityQP2}
    \underset{\xi,\mathbf{u},\mathbf{\lambda}}{\text{argmin}}\, \mathbf{L}_\xi(\mathbf{u},\mathbf{\lambda})\\
        \text{such that } \mathcal{L}_\xi[\mathbf{u}] = \mathbf{f},\\
    \text{where }\mathbf{L}_\xi(\mathbf{u},\mathbf{\lambda}) = ||\mathbf{u} - \mathbf{u}_{data}||^2 + \mathbf{\lambda}^\intercal \left( \mathcal{L}_\xi[\mathbf{u}] - \mathbf{f} \right)
\end{align}
Stationarity of the Karush-Kuhn Tucker conditions requires that the gradient of $ \mathbf{L}_\xi(\mathbf{u},\mathbf{\lambda})$ with respect to $\mathbf{u}$, $\mathbf{\lambda}$ and $\xi$ be zero. This yields the following set of three necessary conditions for a minimizer of~\eqref{eq:equalityQP2}, which we will iteratively solve. If this fixed point iteration converges and all three are satisfied, than one has obtained a minimizer.
\begin{itemize}
    \item \textbf{Forward problem:} The condition $\partial_\mathbf{\lambda} \mathbf{L}_\xi(\mathbf{u},\mathbf{\lambda}) = 0$ requires that the forward problem is solved: $\mathcal{L}_\xi[\mathbf{u}]=\mathbf{f}$. Assuming $\xi$ fixed, one may solve with a Newton-Rhapson method, requiring calculation of the Jacobian of $\mathcal{L}_\xi$, which we denote $J_{\xi,\bm{u}}$. Following the analysis in Section \ref{DDEC}, this problem is guaranteed to be solvable provided $\epsilon L_n < 1$.
    \item \textbf{Adjoint problem:} Enforcing $\partial_\mathbf{u} \mathbf{L}_\xi(\mathbf{u},\mathbf{\lambda}) = 0$ provides the linear adjoint problem $J_{\xi,\bm{u}}^\intercal \mathbf{\lambda} = -2(\mathbf{u} - \mathbf{u}_{data})$ for the Lagrange multipliers. Having solved the forward problem in the previous step, the Jacobian is already available.
    \item \textbf{Model update:} The remaining condition $\partial_\xi \mathbf{L}_\xi(\mathbf{u},\mathbf{\lambda}) = 0$ does not readily admit solution with second-order optimizers due to the neural networks embedded in the nonlinear perturbations of~\eqref{eqn:pertub-primal-Hodge}. It is well-known that neural networks admit a complex optimization landscape due to linear dependence with many suboptimal local minima. With this in mind, we apply a single step of a first-order gradient optimizer instead, providing a small perturbation of the model at each iteration.
\end{itemize}

This process ensures that the physics imposed by the carefully designed model formed in the previous sections hold to machine precision at each iteration, even in scenarios with limited training data. Asymptotically, the added complexity compared to typical gradient-descent approaches is that of solving the forward problem at each iteration, which is generally inexpensive for the low-dimensional models under consideration and converges rapidly. The remaining complexity lies in calculating the relevant derivatives for Newton, which may be simply calculated with the same automatic differentiation used to perform the gradient descent step.

Considering that the assumed model form is nonlinear, it is necessary to train simultaneously on a variety of boundary conditions to learn the nonlinear response across a range of conditions - we must therefore assimilate multiple solutions to the nonlinear problem simultaneously. With that in mind, we present in Algorithm \ref{alg:opt} a batch-training strategy for handling training data $\mathcal{D} = \left\{\bm{u}_{data,i}\right\}_{i=1}^{N_{batches}}$ consisting of $N_{batches}$ solutions.

\RestyleAlgo{plainruled}
\begin{algorithm}[t]\label{alg:opt}
 \KwData{Training solutions $\mathcal{D}$, parameter initialization $\xi_0$, tolerance $\epsilon_{TOL}$}
 \KwResult{$\xi$}
 \For{$e \in \{1,...,\texttt{\textup{epochs}}\}$}{
      $\mathbf{u} \gets \mathbf{0}$\;
 \For{$\mathbf{u}_{data,i} \in \mathcal{D}$}{
      \textit{Solve forward problem}\;
      \While{$||\mathcal{L}_\xi[\mathbf{u}] - f|| > \epsilon_{TOL}$}{
        $J_{\xi,\bm{u}} \gets \nabla_{\mathbf{u}}\mathcal{L}_\xi(\mathbf{u})$\;
        Solve $\Delta = -(J_{\xi,\bm{u}})^{-1} \mathcal{L}_\xi(\mathbf{u})$\;
        $\mathbf{u} \gets \mathbf{u} + \Delta$\;
      }
      \textit{Solve adjoint problem}\;
      $\lambda \gets 2 (J_{\xi,\bm{u}})^{-\intercal}  \left( \mathbf{u}_{data,i} - \mathbf{u}\right)$\;
      \textit{Apply gradient descent update}\;
      $\xi \gets GD(\mathbf{L}_\xi(\mathbf{u},\mathbf{\lambda}),\xi)$\;
  }
  }
 \caption{Application of equality constrained optimizer batched over multiple solutions $\mathbf{u}_{i,data}$. By $GD(l,p)$ we denote a gradient descent update of the parameters $p$ to minimize the loss $l$. For this work we apply the Adam optimizer \cite{kingma2014adam}.}
\end{algorithm}

\section{Numerical results}

\begin{figure}[h]
\centering
\includegraphics[width=0.99\linewidth]{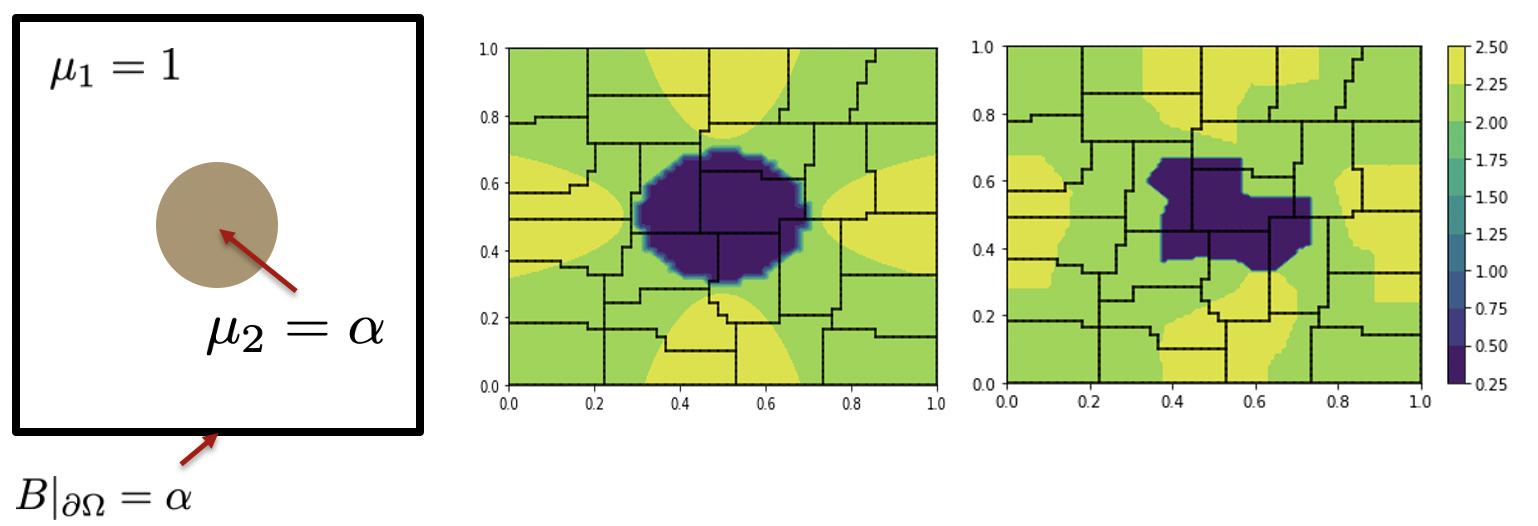}
\caption{We consider a cylindrical inclusion in a 2D domain for steady diffusion problems in subsurface flow and magnetostatics \textit{(left)}, considering diffusion coefficients which vary with the boundary conditions magnitude to obtain a nonlinear response. The high-fidelity solution is coarsened, obtaining a target solution on the corresponding $k-$chains \textit{(center)}. The coarsened graph model is trained to reproduce moments on each subdomain \textit{(right)}. Pictured here is the resulting magnetic field defined on coarsened $0-$chains, with nearest neighbor interpolant of surrogate solution.}
\end{figure}

In the remainder we present results for two canonical $H(div)$ and $H(curl)$ problems from subsurface flow and magnetostatics. In both examples we consider a cylindrical inclusion of radius $a$ embedded within the unit square centered at the origin, whose material properties are prescribed by
\begin{equation}
    \mu_\alpha(\mathbf{x})= 
\begin{cases}
    \alpha              ,& \text{if } ||\mathbf{x}|| < a\\
    1,              & \text{otherwise}
\end{cases}
\end{equation}
Treatment of the material interface at $||\mathbf{x}|| = a$ without interesting spurious oscillations is a hallmark of mimetic discretizations and stems for the exact treatment of interface conditions via Stokes theorem. 

For both problems we will impose boundary conditions by identifying appropriate cochains on the boundary, replacing their corresponding row of the Jacobian matrix with a zero vector with one on the diagonal and setting the desired value on the right hand side. For the introductory nature of this paper this is sufficient, but we note that the imposition of boundary conditions is a rich topic in the discrete exterior calculus literature, with more complex applications requiring a deeper consideration of the interaction between boundary conditions and the discrete exterior calculus \cite{codecasa2014refoundation,kreeft2011mimetic,beltman2018conservative}.

\subsection{Problem 1: Darcy} We consider as training data solutions to the system of equations

\begin{align*}
\mathbf{F} + \mu_\alpha \nabla \phi = 0& \qquad \bm{x} \in \Omega \\
\nabla \cdot \mathbf{F} = f& \qquad \bm{x} \in \Omega\\
\mathbf{F}\cdot \hat{n} = g& \qquad \bm{x} \in \partial \Omega
\end{align*}
obtained via the scheme in \cite{nicolaides2006covolume}. Boundary conditions are imposed upon the 1-cochains encoding the fluxes of $\mathbf{F}$ through subdomain boundaries, and we stress that the resulting DEC method will guarantee that $\int_\Omega \nabla\cdot\mathbf{F} dx = \int_{\partial\Omega} g dA$. For this problem we consider as neural network a dense elu network \cite{clevert2015fast} with two hidden layers of width five, initialized with the He initializer \cite{he2015delving}. We consider a $50^2$ fine mesh to generate training data.

For this problem we will consider two scenarios.\\

\noindent\textbf{Darcy problem 1 (D1):} We take $f=0$ and $g = <1,0> \cdot <1,0>$ consistent with applying a unit horizontal flux and study the effect of varying $\alpha$. For a fixed $\alpha$ the PDE is linear and will not require a contribution from the neural network, and we will use this to gauge the methods ability to recover a PDE discretization consistent with a single solution (i.e. $N_{data} = 1$). Figure \ref{fig:lindarcy} provides a summary and discussion of results.\\

\begin{figure}
\centering
\includegraphics[width=0.8\linewidth]{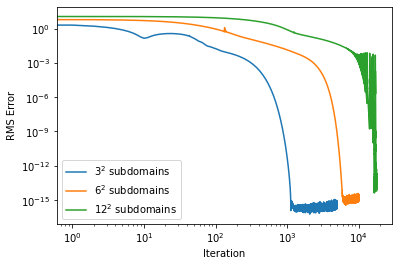}\\

\includegraphics[width=0.32\linewidth]{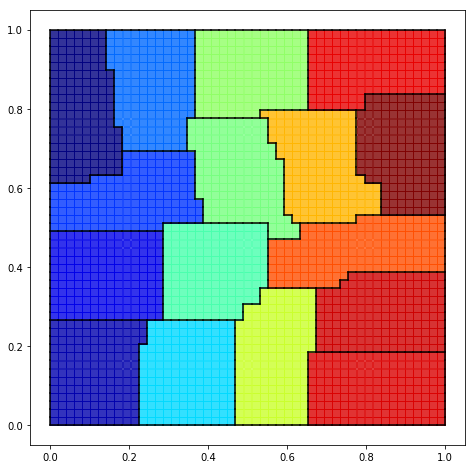}
\includegraphics[width=0.32\linewidth]{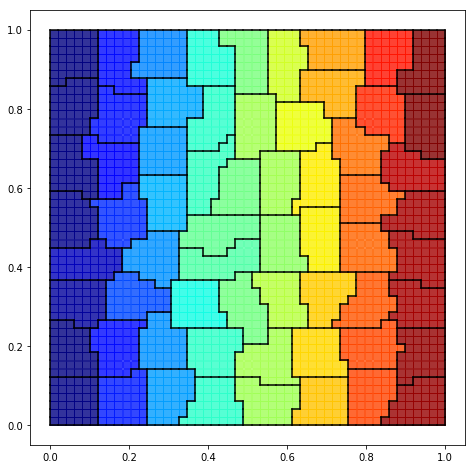}
\includegraphics[width=0.32\linewidth]{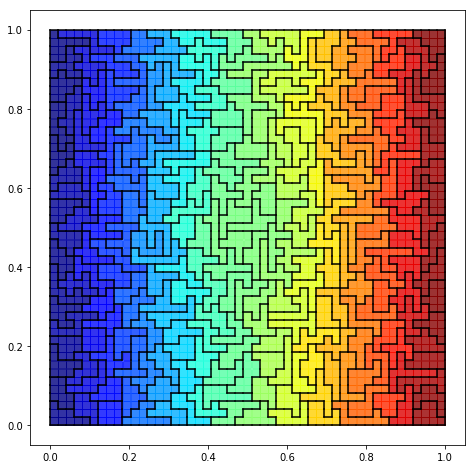}\\
\includegraphics[width=0.32\linewidth]{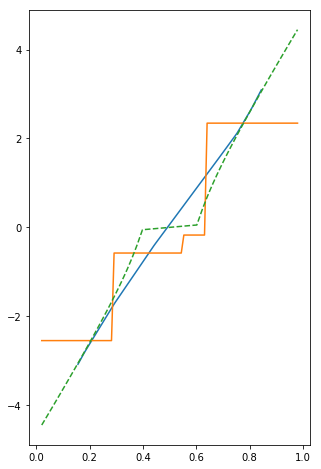}
\includegraphics[width=0.32\linewidth]{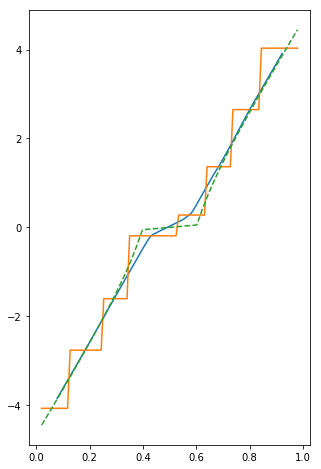}
\includegraphics[width=0.32\linewidth]{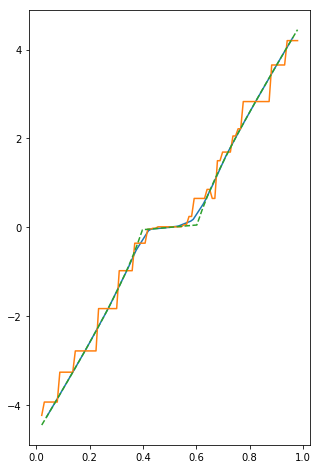}
\caption{Solutions for problem D1. \textit{Top:} RMS error vs number of subdomains in coarsened mesh, using learning rate of $0.05,0.005,0.0005$ for $3^2,6,^2,12^2$ subdomains, respectively, demonstrating ability to capture training data to machine precision. \textit{Center:} piecewise constant plots of pressure for obtained model. \textit{Bottom:} Comparison of pressure along $y=0.5$ for model solution (orange) and Delaunay interpolant (blue) to training solution (dashed green) for increasing resolution, using coarsened subdomain centroids to compute interpolant.}
\label{fig:lindarcy}
\end{figure}

\noindent\textbf{Darcy problem 2 (D2):} We take $f=0$ and $g = <\alpha,0> \cdot <1,0>$ consistent with applying a horizontal flux which scales with the diffusion parameter. This corresponds to a material which becomes more conductive as the current/flux increases. To capture the nonlinear behavior of this problem will require the neural network to learn fluxes which depend upon the magnitude of the potential. Figure \ref{fig:nonlindarcy} provides a summary and discussion of results.

\begin{figure}
\centering
\includegraphics[width=0.79\linewidth]{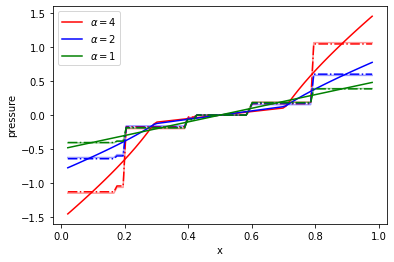}\\
\includegraphics[width=0.79\linewidth]{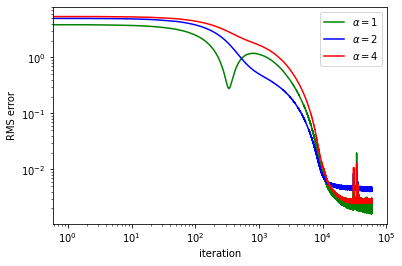}
\caption{Solutions for problem D2. A single nonlinear model is extracted which may be solved for increasing $\alpha$. \textit{Top:} Comparison of pressure along $y=0.5$ for increasing $\alpha$. True solution given by solid line, while learned surrogate is given by dashed line. \textit{Bottom:} Convergence of nonlinear model over three solutions during training.}
\label{fig:nonlindarcy}
\end{figure}

\subsection{Problem 2: Magnetostatics} We consider as training data solutions to the system of equations

\begin{align*}
\nabla \times \mathbf{J} = \mathbf{f} & \qquad \bm{x} \in \Omega \\
\nabla \cdot \mathbf{B} = 0& \qquad \bm{x} \in \Omega\\
\mathbf{B} = \mu_\alpha \mathbf{J}& \qquad \bm{x} \in \Omega\\
\mathbf{B}\cdot \hat{n} = g & \qquad \bm{x} \in \partial \Omega
\end{align*}
obtained via the scheme in \cite{nicolaides1992direct}. For this 2D problem, the magnetic field may be identified with a scalar ($\mathbf{B} = \left<0,0,B(x,y)\right>$). Boundary conditions are imposed upon the 0-cochains encoding the fluxes of $\mathbf{B}$ through subdomain boundaries, and the 1-cochains encoding the magnetic potential on the boundary are fixed to a value of zero. For this problem we will consider only the nonlinear generalization of D2, taking $f=0$ and $g = \alpha$ to obtain a nonlinear material whose permittivity increases with the magnitude of applied magnetic field (Figure \ref{fig:nonlinmagneto}. For this problem, we take as neural network a parametric ReLU activation and a single hidden layer of width ten. In the same manner as the Darcy problem, the resulting model is able to recover a range of $\alpha$'s and the resulting discontinuity in the magnetic field.

\begin{figure}
\centering
\includegraphics[width=0.95\linewidth]{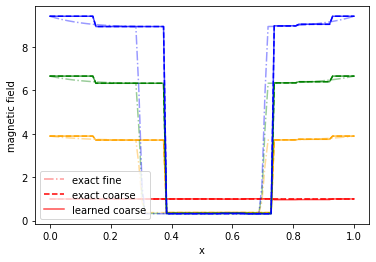}
\caption{Profile of magnetostatics solution along $y=0.5$ line for $5^2$ partitions. The trained model reproduces the coarsened moments of the problem and avoid introduction of any spurious oscillations at the material interface.}
\label{fig:nonlinmagneto}
\end{figure}

\section{Conclusions}

We have presented a new data-driven exterior calculus which parameterizes the classical graph calculus. The resulting framework allows training of discrete exterior calculus operators which incorporates metric information from data while preserving exact sequence structure and invertibility of Hodge Laplacians - both of which are important for handling a range of physical systems. Analysis of the resulting system shows that when nonlinear perturbations are applied to a class of second-order elliptic problems the system remains well-posed under general conditions, allowing the machine learning of nonlinear elliptic systems. Numerical results demonstrate the practical aspects of the approach.

For the sake of introductory exposition, we have restricted ourselves to elementary elliptic problems in the current work. The DDEC framework however may be applied to a broad range of more sophisticated problems; for example, we are currently using it to discover surrogates for semiconductor physics governed by nonlinear drift-diffusion equations (using \cite{charon2020}), and where the resulting network model can be embedded within circuit simulators such as Xyce \cite{gao2020physics,hutchinson2002xyce}. 

Abstractly, the DDEC framework provides a structure-preserving means of parameterizing Dirichlet-to-Neumann maps \cite{sylvester1990dirichlet} governing multiscale physics, generalizing previous works restricted to resistor networks \cite{curtis1991dirichlet} while supporting machine learning tasks. In future work we will provide details regarding how this may be incorporated into a workflow to develop provably stable multiscale models that preserve structure at both fine and coarse scales.

\section*{Acknowledgement}
Sandia National Laboratories is a multimission laboratory managed and operated by National Technology and Engineering Solutions of Sandia, LLC, a wholly owned subsidiary of Honeywell International, Inc., for the U.S. Department of Energy’s National Nuclear Security Administration under contract DE-NA0003525. This paper describes objective technical results and analysis. Any subjective views or opinions that might be expressed in the paper do not necessarily represent the views of the U.S. Department of Energy or the United States Government.

N. Trask has also been supported by the U.S. Department of Energy, Office of Advanced Scientific Computing Research under the Collaboratory on Mathematics and Physics-Informed Learning Machines for Multiscale and Multiphysics Problems (PhILMs) project and the U.S. Department of Energy, Office of Advanced Scientific Computing Research under the Early Career Research Program. A. Huang has been supported under the Sandia National Laboratories Laboratory Directed Research and Development (LDRD) program. The work of X. Hu is partially supported by the National Science Foundation under grant DMS-1812503 and CCF-1934553.

SAND Number: SAND2020-14261 O

\bibliographystyle{plain}
\bibliography{references.bib}

\begin{thebibliography}{10}

\bibitem{AdlerCavanaughHuZikatanov2020}
James~H. Adler, Casey Cavanaugh, Xiaozhe Hu, and Ludmil~T. Zikatanov.
\newblock A finite-element framework for a mimetic finite-difference
  discretization of {{Maxwell}}'s equations.
\newblock {\em arXiv:2012.03148 [cs, math]}, December 2020.

\bibitem{ArnoldFalkWinther2010}
Douglas Arnold, Richard Falk, and Ragnar Winther.
\newblock Finite element exterior calculus: From {{Hodge}} theory to numerical
  stability.
\newblock {\em Bulletin of the American Mathematical Society}, 47(2):281--354,
  2010.

\bibitem{arnold2018finite}
Douglas~N Arnold.
\newblock {\em Finite element exterior calculus}.
\newblock SIAM, 2018.

\bibitem{arnold2007compatible}
Douglas~N Arnold, Pavel~B Bochev, Richard~B Lehoucq, Roy~A Nicolaides, and
  Mikhail Shashkov.
\newblock {\em Compatible spatial discretizations}, volume 142.
\newblock Springer Science \& Business Media, 2007.

\bibitem{ArnoldFalkWinther2006b}
Douglas~N. Arnold, Richard~S. Falk, and Ragnar Winther.
\newblock Finite element exterior calculus, homological techniques, and
  applications.
\newblock {\em Acta Numerica}, 15:1--155, May 2006.

\bibitem{baker2019workshop}
Nathan Baker, Frank Alexander, Timo Bremer, Aric Hagberg, Yannis Kevrekidis,
  Habib Najm, Manish Parashar, Abani Patra, James Sethian, Stefan Wild, et~al.
\newblock Workshop report on basic research needs for scientific machine
  learning: Core technologies for artificial intelligence.
\newblock Technical report, USDOE Office of Science (SC), Washington, DC
  (United States), 2019.

\bibitem{bamberg-sternberg-1991}
Paul Bamberg and Shlomo Sternberg.
\newblock {\em A Course in Mathematics for Students of Physics}, volume~2.
\newblock 1991.

\bibitem{BarbarossaSardellitti2020}
Sergio Barbarossa and Stefania Sardellitti.
\newblock Topological {{Signal Processing}} over {{Simplicial Complexes}}.
\newblock {\em arXiv:1907.11577 [eess]}, March 2020.

\bibitem{barth2006role}
Timothy Barth.
\newblock On the role of involutions in the discontinuous galerkin
  discretization of maxwell and magnetohydrodynamic systems.
\newblock In {\em Compatible spatial discretizations}, pages 69--88. Springer,
  2006.

\bibitem{battaglia2016interaction}
Peter Battaglia, Razvan Pascanu, Matthew Lai, Danilo Jimenez~Rezende, et~al.
\newblock Interaction networks for learning about objects, relations and
  physics.
\newblock {\em Advances in neural information processing systems},
  29:4502--4510, 2016.

\bibitem{BattistonCencettiIacopiniLatoraLucasPataniaYoungPetri2020}
Federico Battiston, Giulia Cencetti, Iacopo Iacopini, Vito Latora, Maxime
  Lucas, Alice Patania, Jean-Gabriel Young, and Giovanni Petri.
\newblock Networks beyond pairwise interactions: {{Structure}} and dynamics.
\newblock {\em Physics Reports}, June 2020.

\bibitem{beltman2018conservative}
Ren{\'e} Beltman, MJH Anthonissen, and Barry Koren.
\newblock Conservative polytopal mimetic discretization of the incompressible
  navier--stokes equations.
\newblock {\em Journal of Computational and Applied Mathematics}, 340:443--473,
  2018.

\bibitem{bloch1945electromechanical}
A~Bloch.
\newblock Electromechanical analogies and their use for the analysis of
  mechanical and electromechanical systems.
\newblock {\em Journal of the Institution of Electrical Engineers-Part I:
  General}, 92(52):157--169, 1945.

\bibitem{bochev2001matching}
P~Bochev and A~Robinson.
\newblock Matching algorithms with physics: exact sequences of finite element
  spaces.
\newblock {\em Collected Lectures on the Preservation of Stability Under
  Discretization, edited by D. Estep and S. Tavener, SIAM, Philadelphia}, 2001.

\bibitem{BochevHuEtAl2008_AlgebraicMultigridApproachBased}
Pavel~B Bochev, Jonathan~J Hu, Christopher~M Siefert, and Raymond~S Tuminaro.
\newblock An algebraic multigrid approach based on a compatible gauge
  reformulation of maxwell's equations.
\newblock {\em SIAM Journal on Scientific Computing}, 31(1):557--583, 2008.

\bibitem{bochev2006principles}
Pavel~B Bochev and James~M Hyman.
\newblock Principles of mimetic discretizations of differential operators.
\newblock In {\em Compatible spatial discretizations}, pages 89--119. Springer,
  2006.

\bibitem{bossavit1998computational}
Alain Bossavit.
\newblock {\em Computational electromagnetism: variational formulations,
  complementarity, edge elements}.
\newblock Academic Press, 1998.

\bibitem{bossavit2001stiff}
Alain Bossavit.
\newblock " stiff" problems in eddy-current theory and the regularization of
  maxwell's equations.
\newblock {\em IEEE transactions on magnetics}, 37(5):3542--3545, 2001.

\bibitem{breedveld1985multibond}
Peter~C Breedveld.
\newblock Multibond graph elements in physical systems theory.
\newblock {\em Journal of the Franklin Institute}, 319(1-2):1--36, 1985.

\bibitem{chang2016compositional}
Michael~B Chang, Tomer Ullman, Antonio Torralba, and Joshua~B Tenenbaum.
\newblock A compositional object-based approach to learning physical dynamics.
\newblock {\em arXiv preprint arXiv:1612.00341}, 2016.

\bibitem{chen2015electrical}
Qun Chen, Rong-Huan Fu, and Yun-Chao Xu.
\newblock Electrical circuit analogy for heat transfer analysis and
  optimization in heat exchanger networks.
\newblock {\em Applied Energy}, 139:81--92, 2015.

\bibitem{chung1997spectral}
Fan~RK Chung and Fan~Chung Graham.
\newblock {\em Spectral graph theory}.
\newblock Number~92. American Mathematical Soc., 1997.

\bibitem{clevert2015fast}
Djork-Arn{\'e} Clevert, Thomas Unterthiner, and Sepp Hochreiter.
\newblock Fast and accurate deep network learning by exponential linear units
  (elus).
\newblock {\em arXiv preprint arXiv:1511.07289}, 2015.

\bibitem{codecasa2014refoundation}
Lorenzo Codecasa.
\newblock Refoundation of the cell method using augmented dual grids.
\newblock {\em IEEE transactions on magnetics}, 50(2):497--500, 2014.

\bibitem{curtis1991dirichlet}
Edward~B Curtis and James~A Morrow.
\newblock The dirichlet to neumann map for a resistor network.
\newblock {\em SIAM Journal on Applied Mathematics}, 51(4):1011--1029, 1991.

\bibitem{dafermos1986quasilinear}
Constantine~M Dafermos.
\newblock Quasilinear hyperbolic systems with involutions.
\newblock {\em Archive for Rational Mechanics and Analysis}, 94(4):373--389,
  1986.

\bibitem{desbrun-hirani-leok-marsden-2005}
Mathieu Desbrun, Anil~N. Hirani, Melvin Leok, and Jerrold~E. Marsden.
\newblock Discrete exterior calculus.
\newblock 2005.

\bibitem{Eckmann1945}
Beno Eckmann.
\newblock Harmonische funktionen und randwertaufgaben in einem komplex.
\newblock {\em Commentarii Mathematici Helvetici}, 17(1):240--255, 1944.

\bibitem{firestone1933new}
Floyd~A Firestone.
\newblock A new analogy between mechanical and electrical systems.
\newblock {\em The Journal of the Acoustical Society of America},
  4(3):249--267, 1933.

\bibitem{gao2020physics}
Xujiao Gao, Andy Huang, Nathaniel Trask, and Shahed Reza.
\newblock Physics-informed graph neural network for circuit compact model
  development.
\newblock In {\em 2020 International Conference on Simulation of Semiconductor
  Processes and Devices (SISPAD)}, pages 359--362. IEEE.

\bibitem{gilbarg2015elliptic}
David Gilbarg and Neil~S Trudinger.
\newblock {\em Elliptic partial differential equations of second order}.
\newblock springer, 2015.

\bibitem{gori2005new}
Marco Gori, Gabriele Monfardini, and Franco Scarselli.
\newblock A new model for learning in graph domains.
\newblock In {\em Proceedings. 2005 IEEE International Joint Conference on
  Neural Networks, 2005.}, volume~2, pages 729--734. IEEE, 2005.

\bibitem{hamilton2017inductive}
Will Hamilton, Zhitao Ying, and Jure Leskovec.
\newblock Inductive representation learning on large graphs.
\newblock In {\em Advances in neural information processing systems}, pages
  1024--1034, 2017.

\bibitem{hamilton2017representation}
William~L Hamilton, Rex Ying, and Jure Leskovec.
\newblock Representation learning on graphs: Methods and applications.
\newblock {\em arXiv preprint arXiv:1709.05584}, 2017.

\bibitem{he2015delving}
Kaiming He, Xiangyu Zhang, Shaoqing Ren, and Jian Sun.
\newblock Delving deep into rectifiers: Surpassing human-level performance on
  imagenet classification.
\newblock In {\em Proceedings of the IEEE international conference on computer
  vision}, pages 1026--1034, 2015.

\bibitem{hutchinson2002xyce}
S~Hutchinson, E~Keiter, R~Hoekstra, H~Watts, A~Waters, T~Russo, R~Schells,
  S~Wix, and C~Bogdan.
\newblock The xyce™ parallel electronic simulator--an overview.
\newblock In {\em Parallel Computing: Advances and Current Issues}, pages
  165--172. World Scientific, 2002.

\bibitem{JiangLimYaoYe2011}
Xiaoye Jiang, Lek-Heng Lim, Yuan Yao, and Yinyu Ye.
\newblock Statistical ranking and combinatorial {{Hodge}} theory.
\newblock {\em Mathematical Programming}, 127(1):203--244, 2011.

\bibitem{jiang2011statistical}
Xiaoye Jiang, Lek-Heng Lim, Yuan Yao, and Yinyu Ye.
\newblock Statistical ranking and combinatorial hodge theory.
\newblock {\em Mathematical Programming}, 127(1):203--244, 2011.

\bibitem{karypis1997metis}
George Karypis and Vipin Kumar.
\newblock Metis: A software package for partitioning unstructured graphs,
  partitioning meshes, and computing fill-reducing orderings of sparse
  matrices.
\newblock 1997.

\bibitem{kingma2014adam}
Diederik~P Kingma and Jimmy Ba.
\newblock Adam: A method for stochastic optimization.
\newblock {\em arXiv preprint arXiv:1412.6980}, 2014.

\bibitem{koenig1960linear}
HE~Koenig and WA~Blackwell.
\newblock Linear graph theory-a fundamental engineering discipline.
\newblock {\em IRE Transactions on Education}, 3(2):42--49, 1960.

\bibitem{kreeft2011mimetic}
Jasper Kreeft, Artur Palha, and Marc Gerritsma.
\newblock Mimetic framework on curvilinear quadrilaterals of arbitrary order.
\newblock {\em arXiv preprint arXiv:1111.4304}, 2011.

\bibitem{lagaris1998artificial}
Isaac~E Lagaris, Aristidis Likas, and Dimitrios~I Fotiadis.
\newblock Artificial neural networks for solving ordinary and partial
  differential equations.
\newblock {\em IEEE transactions on neural networks}, 9(5):987--1000, 1998.

\bibitem{Lim2020}
Lek-Heng Lim.
\newblock Hodge {{Laplacians}} on {{Graphs}}.
\newblock {\em SIAM Review}, 62(3):685--715, January 2020.

\bibitem{MaleticRajkovic2012}
S.~Maleti{\'c} and M.~Rajkovi{\'c}.
\newblock Combinatorial {{Laplacian}} and entropy of simplicial complexes
  associated with complex networks.
\newblock {\em The European Physical Journal Special Topics}, 212(1):77--97,
  September 2012.

\bibitem{MuhammadEgerstedt2006}
Abubakr Muhammad and Magnus Egerstedt.
\newblock Control {{Using Higher Order Laplacians}} in {{Network Topologies}}.
\newblock page~15, 2006.

\bibitem{charon2020}
Lawrence Musson, Xujiao Gao, Mihai Negoita, Andy Huang, and Gary~L Hennigan.
\newblock Charon: A radiation aware massively parallel tcad modeling code.
\newblock Technical report, Sandia National Lab.(SNL-NM), Albuquerque, NM
  (United States), 2018.

\bibitem{nedelec1980mixed}
Jean-Claude N{\'e}d{\'e}lec.
\newblock Mixed finite elements in {$\mathbb{R}^3$}.
\newblock {\em Numerische Mathematik}, 35(3):315--341, 1980.

\bibitem{nicolaides1992direct}
R.~A. Nicolaides.
\newblock Direct {{Discretization}} of {{Planar Div}}-{{Curl Problems}}.
\newblock {\em SIAM Journal on Numerical Analysis}, 29(1):32--56, February
  1992.

\bibitem{nicolaides2006covolume}
Roy~A Nicolaides and Kathryn~A Trapp.
\newblock Covolume discretization of differential forms.
\newblock In {\em Compatible spatial discretizations}, pages 161--171.
  Springer, 2006.

\bibitem{oh2012design}
Kwang~W Oh, Kangsun Lee, Byungwook Ahn, and Edward~P Furlani.
\newblock Design of pressure-driven microfluidic networks using electric
  circuit analogy.
\newblock {\em Lab on a Chip}, 12(3):515--545, 2012.

\bibitem{ohlberger2015reduced}
Mario Ohlberger and Stephan Rave.
\newblock Reduced basis methods: Success, limitations and future challenges.
\newblock {\em arXiv preprint arXiv:1511.02021}, 2015.

\bibitem{quarteroni2015reduced}
Alfio Quarteroni, Andrea Manzoni, and Federico Negri.
\newblock {\em Reduced basis methods for partial differential equations: an
  introduction}, volume~92.
\newblock Springer, 2015.

\bibitem{raissi2019physics}
Maziar Raissi, Paris Perdikaris, and George~E Karniadakis.
\newblock Physics-informed neural networks: A deep learning framework for
  solving forward and inverse problems involving nonlinear partial differential
  equations.
\newblock {\em Journal of Computational Physics}, 378:686--707, 2019.

\bibitem{raissi2017physics}
Maziar Raissi, Paris Perdikaris, and George~Em Karniadakis.
\newblock Physics informed deep learning (part i): Data-driven solutions of
  nonlinear partial differential equations.
\newblock {\em arXiv preprint arXiv:1711.10561}, 2017.

\bibitem{rosenblatt1961principles}
Frank Rosenblatt.
\newblock Principles of neurodynamics. perceptrons and the theory of brain
  mechanisms.
\newblock Technical report, Cornell Aeronautical Lab Inc Buffalo NY, 1961.

\bibitem{ScamanVirmaux2019}
Kevin Scaman and Aladin Virmaux.
\newblock Lipschitz regularity of deep neural networks: Analysis and efficient
  estimation.
\newblock {\em arXiv:1805.10965 [cs, stat]}, October 2019.

\bibitem{SchaubBensonHornLippnerJadbabaie2020}
Michael~T. Schaub, Austin~R. Benson, Paul Horn, Gabor Lippner, and Ali
  Jadbabaie.
\newblock Random {{Walks}} on {{Simplicial Complexes}} and the {{Normalized
  Hodge}} 1-{{Laplacian}}.
\newblock {\em SIAM Review}, 62(2):353--391, January 2020.

\bibitem{smale-1972}
S.~Smale.
\newblock On the mathematical foundations of electrical circuit theory.
\newblock {\em J. Differential Geom.}, 7(1-2):193--210, 1972.

\bibitem{smith2002synthesis}
Malcolm~C Smith.
\newblock Synthesis of mechanical networks: the inerter.
\newblock {\em IEEE Transactions on automatic control}, 47(10):1648--1662,
  2002.

\bibitem{sylvester1990dirichlet}
John Sylvester and Gunther Uhlmann.
\newblock The dirichlet to neumann map and applications.
\newblock In {\em Inverse problems in partial differential equations},
  volume~42, page 101. SIAM Publications, Philadelphia, 1990.

\bibitem{TorresBianconi2020}
Joaqu{\'i}n~J. Torres and Ginestra Bianconi.
\newblock Simplicial complexes: Higher-order spectral dimension and dynamics.
\newblock {\em Journal of Physics: Complexity}, 1(1):015002, May 2020.

\bibitem{vlassis2020geometric}
Nikolaos Vlassis, Ran Ma, and WaiChing Sun.
\newblock Geometric deep learning for computational mechanics part i:
  Anisotropic hyperelasticity.
\newblock {\em arXiv preprint arXiv:2001.04292}, 2020.

\bibitem{wang2020understanding}
Sifan Wang, Yujun Teng, and Paris Perdikaris.
\newblock Understanding and mitigating gradient pathologies in physics-informed
  neural networks.
\newblock {\em arXiv preprint arXiv:2001.04536}, 2020.

\bibitem{wang2020and}
Sifan Wang, Xinling Yu, and Paris Perdikaris.
\newblock When and why pinns fail to train: A neural tangent kernel
  perspective.
\newblock {\em arXiv preprint arXiv:2007.14527}, 2020.

\bibitem{weyl-1923}
Hermann Weyl.
\newblock Repartición de corriente en una red conductora.
\newblock {\em Revista Matemática Hispano-Americana}, 5:153--164, 1923.

\bibitem{wilf-hilary-1972}
Herbert~S. Wilf and Frank Hilary.
\newblock {\em "Mathematical aspects of electrical network analysis"}, volume~3
  of {\em SIAM-AMS proceedings}.
\newblock American Mathematical Society, 1971.

\bibitem{yee1966numerical}
Kane Yee.
\newblock Numerical solution of initial boundary value problems involving
  maxwell's equations in isotropic media.
\newblock {\em IEEE Transactions on antennas and propagation}, 14(3):302--307,
  1966.

\end{thebibliography}

\end{document}